\pdfoutput=1
\documentclass{amsart}
\usepackage{amsmath,amssymb}
\usepackage{epsfig}
\usepackage{amsthm}

\newcommand{\N}{\mathbb{N}}

\newcommand{\R}{\mathbb{R}}
\newcommand{\C}{\mathbb{C}}

\newcommand{\overbar}[1]{\mkern 1.0mu\overline{\mkern-1.0mu#1\mkern-1.0mu}\mkern 1.0mu}

\newtheorem{theorem}{Theorem}[section]
\newtheorem{lemma}[theorem]{Lemma}

\newtheorem{remark}[theorem]{Remark}

\def\dS{\,{\rm d}S}

\DeclareMathOperator*{\argmin}{arg\,min}

\begin{document}
\title[Smoothened complete electrode model]{Smoothened complete electrode model}

\author{Nuutti Hyv\"onen}
\address{Aalto University, Department of Mathematics and Systems Analysis, P.O. Box 11100, FI-00076 Aalto, Finland} 
\email{nuutti.hyvonen@aalto.fi}

\author{Lauri Mustonen}
\address{Aalto University, Department of Mathematics and Systems Analysis, P.O. Box 11100, FI-00076 Aalto, Finland} 
\email{lauri.mustonen@aalto.fi}

\thanks{This work was supported by the Academy of Finland (decision 267789) and the Finnish Foundation for Technology Promotion TES}



\begin{abstract}
This work reformulates the complete electrode model of electrical impedance tomography in order to enable more efficient numerical solution. The model traditionally assumes constant contact conductances on all electrodes, which leads to a discontinuous Robin boundary condition since the gaps between the electrodes can be described by vanishing conductance. As a consequence, the regularity of the electromagnetic potential is limited to less than two square-integrable weak derivatives, which negatively affects the convergence of,~e.g.,~the finite element method. In this paper, a smoothened model for the boundary conductance is proposed, and the unique solvability and improved regularity of the ensuing boundary value problem are proven. Numerical experiments demonstrate that the proposed model is both computationally feasible and also compatible with real-world measurements. In particular, the new model allows faster convergence of the finite element method. 
\end{abstract}

\maketitle

\section{Introduction}
\label{sec:intro}

{\em Electrical impedance tomography} (EIT) is a noninvasive imaging technique based on controlling and measuring electric currents and voltages on the surface of the imaged object.
The aim is to reconstruct the electrical conductivity (or admittivity, resistivity, or impedivity) inside the object.
Applications of EIT include biomedical imaging \cite{Bayford06}, nondestructive testing \cite{Karhunen10}, and process tomography \cite{Seppanen08}.
The reconstruction task is an extensively studied inverse problem for which both direct and iterative methods have been proposed \cite{Kaipio00,Lionheart04,Uhlmann09}.

Several mathematical models have been applied to incorporating boundary conditions in the forward problem of EIT including those discussed in \cite{Cheng89,Somersalo92}.
The simplest one is the continuum model, which assumes a (typically smooth) Neumann boundary condition.
This is useful for theoretical considerations and convenient in numerical computations, but does not typically result in accurate reconstructions because practical measurement setups employ a finite number of electrodes, which is not taken into account in the continuum model.
The point electrode model is mainly useful when the electrodes are small and the reconstruction is based on difference measurements \cite{Hanke11b}.
The so called shunt model correctly models the geometry of the electrodes but neglects the thin resistive layer that may appear at the contact between the electrodes and the object.
The presence of this layer is included in the {\em complete electrode model} (CEM) that has become the standard for computing reconstructions in practical applications \cite{Cheng89}.
Loosely speaking, the shunt model can be regarded as the limit of the CEM when the contact resistances tend to zero \cite{Darde16}.

When the CEM is employed, the contact conductances (or admittances, resistances, or impedances) are usually not known, but they are estimated along with the interior conductivity \cite{Heikkinen02}. One may also simultaneously reconstruct the electrode locations and the shape of the imaged object \cite{Darde13a,Darde13b,Nissinen11a}.
However, even if all these parameters were known, an inherent property of the traditional CEM is that the employed ``discontinuous'' Robin-type boundary condition causes the regularity of the electromagnetic potential to be limited, namely, it is of the Sobolev class $H^{2-\epsilon}$, $\epsilon > 0$.
The same conclusion applies to an even greater extent to the shunt model for which the potential only exhibits $H^{3/2-\epsilon}$-regularity.
Regarding numerical computations, the lack of smoothness makes it difficult to construct forward solvers that converge fast; in particular, the efficient use of  {\em finite element method}s (FEM) of higher order is prevented. This is an issue for iterative reconstruction algorithms that require repetitive and accurate
solutions of the forward problem.
In addition, certain quantities that are derived from forward solutions of the CEM, such as those needed when computing shape derivatives of electrode measurements (cf.~\cite{Darde13a,Darde13b}), are expected to suffer in accuracy 
even when first-order FEM is used 
\cite[Remark~2.4]{Darde13b}.

To overcome these problems arising from the discontinuity of a conductance coefficient in the Robin boundary condition, we propose a smoothened version of the CEM that 
exploits nonhomogeneous conductances and is more suitable for numerical computations.
It is shown that arbitrarily high regularity for the (interior) electromagnetic potential in EIT can be achieved while only slightly deviating from the standard CEM, assuming the conductivity and the object boundary are smooth enough.
Moreover, it is numerically demonstrated that 
forward computations exhibit faster convergence (both in practice and asymptotically). 
According to our preliminary tests, a version of the smoothened model is in approximately as good agreement with experimental data as the standard CEM and reconstructions based on the two models are almost indistinguishable.

It should be emphasized that we do not claim that the smoothened CEM is a more appropriate model for EIT from the standpoint of the physical phenomena occurring at the electrode-object interface. Our assertion is merely that the new model is computationally more efficient and predicts the {\em electrode measurements} of EIT with accuracy comparable to the traditional CEM. 
How well the new model --- or the new family of models --- predicts the behavior of the electromagnetic potential in the interior of the imaged object remains an open question; the same is actually true also for the traditional CEM as its performance has only been validated in regard to electrode measurements \cite{Cheng89}. (Observe that there is no actual reason to expect that the traditional assumption of having a constant contact conductance on each electrode is completely accurate either. In fact, the standard CEM can be viewed as a special case of the smoothened CEM.)

Although we 
only discuss EIT, it is worth noting that same kinds of electrode models can be used in, e.g., {\em electrical capacitance tomography} (ECT) \cite{Fang16} and {\em electroencephalography} (EEG) \cite{Pursiainen12}.
In particular, since ECT is mathematically equivalent to EIT, our theoretical results apply directly to ECT as well.

This text is organized as follows.
In the next section, we introduce the smoothen\-ed CEM for EIT and prove the unique solvability of the corresponding elliptic boundary value problem.
Our main theoretical result addressing the regularity of the electromagnetic potential is also formulated and proven in that section.
Section~\ref{sec:Frechet} reviews the Fr\'echet differentiability of electrode potentials with respect to the shape of the imaged object, which provides an example of a setting where the new smoothened model clearly prevails.
Numerical examples demonstrating the improved convergence are presented in Section~\ref{sec:numerics}, where the smoothened model is also compared to the traditional CEM and to experimental data. In addition, example reconstructions based on water tank measurements are presented. Finally, conclusions are drawn in Section~\ref{sec:conclusion}.

\section{Smoothened complete electrode model}
\label{sec:CEM}

A physical body imaged by EIT is modeled as a bounded Lipschitz domain $\Omega \subset \R^n$, $n=2$ or $3$. The boundary $\partial \Omega$ is partially covered by $M \in \N \setminus \{ 1 \}$ electrodes $\{ E_m\}_{m=1}^M$ that are identified with nonempty, connected, open surface patches and assumed to be well-separated,~i.e.,~$\overline{E}_m\cap \overline{E}_l = \emptyset$ if $m\not= l$.  We denote $E = \cup E_m$. The net currents $I_m\in\C $, $m=1, \dots, M$, are driven through the corresponding electrodes and the resulting constant electrode potentials $U_m \in \C$, $m=1, \dots, M$, are measured. Due to the conservation of electric charge and under the reasonable assumption that there are no sinks or sources inside $\Omega$, any realizable current pattern $I = [I_1,\dots,I_M]^{\rm T}$ belongs to the subspace
\[ 
\C^M_\diamond \, := \, \Big\{J \in\C^M\,\Big|\, \sum_{m=1}^M J_m = 0\Big\}.
\]
The electrode potential vector $U = [U_1,\dots,U_M]^{\rm T}$ is identified with
\begin{equation}
\label{eq:piecewise}
U \, = \, \sum_{m=1}^M U_m \chi_m 
\end{equation}
where $\chi_m$ is the characteristic function of $E_m \subset \partial \Omega$. Whether $U$ refers to such a piecewise constant function supported on $\overline{E}$ or to a vector of $\C^M$ should be clear from the context; in particular, under an integral a capital letter always refers to a piecewise constant function vanishing in between the electrodes.  The real-symmetric admittivity distribution $\sigma\in L^{\infty}(\Omega, \C^{n \times n})$ inside $\Omega$ is assumed to satisfy
\begin{equation}
\label{eq:sigma}
{\rm Re} (\sigma \xi \cdot \overbar{\xi}) \geq \varsigma_- \| \xi \|_2^2,
\end{equation}
for all $\xi \in \C^n$ almost everywhere in $\Omega$ with $\varsigma_->0$ being some positive constant.
In other words, the imaged body is allowed to be characterized by anisotropic conductivity and permittivity but these coefficients are required to be symmetric and the conductivity, in addition, strictly positive definite (cf., e.g., \cite{Vauhkonen97}).
Here and in what follows, $\|  \cdot \|_2$ denotes the Euclidean norm of a finite-dimensional vector.

The CEM is a mathematical model that accurately predicts real-life EIT measurements, i.e.,
its validity has been confirmed in regard to data collected at electrodes \cite{Cheng89}. We consider a nonstandard formulation of the CEM: The electromagnetic potential $u$ inside $\Omega$ and the piecewise constant electrode potential $U$ satisfy
\begin{equation}
\label{eq:cemeqs}
\begin{array}{ll}
\displaystyle{\nabla \cdot(\sigma\nabla u) = 0 \qquad}  &{\rm in}\;\; \Omega, \\[6pt] 
{\displaystyle {\nu\cdot\sigma\nabla u} = \zeta (U - u) } \qquad &{\rm on}\;\; \partial \Omega, \\[2pt] 
{\displaystyle \int_{E_m}\nu\cdot\sigma\nabla u\,{\rm d}S} = I_m, \qquad & m=1,\ldots,M, \\[4pt]
\end{array}
\end{equation}
interpreted in the weak sense. Here, $ \nu \in L^\infty(\partial \Omega, \R^n)$ denotes the exterior unit normal of $\partial\Omega$ and $\zeta \in L^\infty(\partial \Omega)$ describes the contact admittance over $\partial \Omega$. The gaps between the electrodes can be characterized by vanishing admittance. Moreover, the conductance,~i.e.,~the real part of the admittance, cannot be negative, and to be able to drive currents through the electrodes, the conductance must not vanish everywhere on any of the electrodes. To summarize, it is physically reasonable to assume 
\begin{equation}
\label{eq:zeta}
{\rm Re} (\zeta) \geq 0,
\qquad \zeta_{\partial \Omega \setminus{\overline{E}}} \equiv 0, \qquad
{\rm Re}\big(\zeta|_{E_m}\big) \not\equiv 0 
\end{equation}
for all $m=1, \dots, M$ in the topology of $L^\infty(\partial \Omega)$. Take note that the second assumption on $\zeta$ reduces the second equation of \eqref{eq:cemeqs} into a homogeneous Neumann condition on $\partial \Omega \setminus{\overline{E}}$,~i.e.,~no current flows through the object boundary in between the electrodes. 

A  physical justification of \eqref{eq:cemeqs} can be found in \cite{Cheng89}, where the second condition is divided into two parts as
\begin{equation}
\label{eq:alternate}
\begin{array}{ll}
{\displaystyle{\nu\cdot\sigma\nabla u} = 0 }\qquad &{\rm on}\;\;\partial\Omega\setminus\overbar{E},\\[6pt] 
{\displaystyle {u + z_m \nu\cdot\sigma\nabla u} = U_m } \qquad &{\rm on}\;\; E_m, \quad m=1, \dots, M, 
\end{array}
\end{equation}
and the contact impedances $z_m := (1/\zeta)|_{E_m}$, $m=1, \dots, M$, are assumed to be constants. Although the CEM has previously been analyzed also for nonconstant contact impedances (cf.~\cite{Hyvonen04,Winkler14b}), to the authors' knowledge all previous mathematical works on the CEM assume the impedances are bounded away from infinity. As the assumptions \eqref{eq:zeta} allow the contact admittances to vanish on some subsets of the electrodes, the unique solvability of \eqref{eq:cemeqs} does not directly follow from previous analyses but a bit of extra work is required.

We look for the solution of \eqref{eq:cemeqs} in the quotient space $\mathcal{H}^1$, with the definition
\begin{align*}
 \mathcal{H}^s &:= \big\{ \{ (v + c, V + c {\bf 1}) \, | \, c \in \C \} \, \big| \, (v, V) \in H^s(\Omega)\oplus \C^M \big\} \\[1mm]
& \simeq \Big\{ \big\{ \big (v + c, \sum_m (V_m + c) \chi_m \big) \, | \, c \in \C \big\} \, \Big| \, (v, V) \in H^s(\Omega)\oplus \C^M \Big\}
\end{align*}
for $s \in \R$ with ${\bf 1} := [1, \dots, 1]^{\rm T}\in \R^M$. 
Here, ``$\simeq$'' is to be understood via an isomorphic identification of vectors and piecewise constant functions on the electrodes. The use of a quotient structure reflects the freedom in the choice of the ground level of potential: All elements of $H^s(\Omega)\oplus \C^M$ that differ by an additive constant are identified as an equivalence class. In particular, when the second component of an element of $\mathcal{H}^s$ is interpreted as a piecewise constant function on the electrodes, the additive constant is also supported on $\overline{E}$.

 To prove the unique solvability of \eqref{eq:cemeqs}, notice first that the standard quotient norm for $\mathcal{H}^1$ is defined by 
$$
\|(v,V)\|_{\mathcal{H}^1} = \inf_{c\in\C}\Big( \|v-c\|_{H^1(\Omega)}^2 + \| V - c {\bf 1}\|_2^2 \Big)^{1/2}.
$$
Moreover, by following the line of reasoning in \cite{Somersalo92}, one sees that the variational formulation of \eqref{eq:cemeqs} is to find $(u,U) \in \mathcal{H}^1$ that satisfies
\begin{equation}
\label{eq:weak}
B\big((u,U),(v,V)\big)  \,=  \, I\cdot \overbar{V} \qquad {\rm for} \ {\rm all} \ (v,V) \in  \mathcal{H}^1,
\end{equation}
where ``$\,\cdot\,$'' denotes the real inner product and the sesquilinear form $B: \mathcal{H}^1 \times \mathcal{H}^1 \to \C$ is defined by
\begin{equation}
\label{eq:sesqui}
B\big((w,W),(v,V)\big) = \int_\Omega \sigma\nabla w\cdot \nabla \overbar{v} \,{\rm d}x + \int_{\partial \Omega} \zeta (W-w)(\overbar{V}-\overbar{v})\,{\rm d}S,
\end{equation}
with $W, V \in \C^M$ identified with the corresponding piecewise constant functions.

\begin{lemma}
\label{lemma:coerc}
Under the assumptions \eqref{eq:sigma} and \eqref{eq:zeta}, the sesquilinear form $B: \mathcal{H}^1 \times \mathcal{H}^1 \to \C$ is well-defined, bounded and coercive, that is,
$$
\big| B\big((w,W), (v,V) \big) \big| \, \leq \, C \|(w,W) \|_{\mathcal{H}^1}  \|(v,V) \|_{\mathcal{H}^1}   
$$
and
$$
{\rm Re} \Big( B \big((v,V), (v,V) \big) \Big) \, \geq \, c \|(v,V) \|_{\mathcal{H}^1}^2, 
$$
where $c, C > 0$ do not depend on $(w,W), (v,V) \in \mathcal{H}^1$.
\end{lemma}

\begin{proof}
First of all, due to the second condition of \eqref{eq:zeta}, there is no ambiguity in the definition of $B$ on $\mathcal{H}^1 \times \mathcal{H}^1$,~i.e.,~the value $B((w,W),(v,V))$ does not depend on the particular representatives of the equivalence classes $(w,W), (v,V) \in \mathcal{H}^1$. Moreover, the continuity of $B$ can be proved by following the argumentation in \cite[proof of Lemma~2.5]{Hyvonen04}.

Since $\zeta$ does not vanish identically almost everywhere on any electrode (cf.~\eqref{eq:zeta}), there exist open subsets $e_m \subset E_m$, $m=1,\dots,M$, of nonzero measure and a constant $\zeta_- > 0$ such that
\begin{equation}
\label{eq:trueelec}
{\rm Re} ( \zeta) \geq \zeta_-  \qquad {\rm a.e.} \ {\rm on} \ e := \bigcup_{m=1}^M e_m. 
\end{equation}
To deduce the coercivity, note first that
$$
\| V - c {\bf 1}\|_2^2 \leq C \| V - c \|_{L^2(e)}^2 
\leq C \Big(\| V - v \|_{L^2(e)}^2 + \| v - c \|_{L^2(\partial \Omega)}^2 \Big)
$$
due to the triangle inequality. Hence, by the trace theorem,
\begin{align}
\label{eq:almostcoer}
\| (v,V) \|_{\mathcal{H}^1}^2 
&\leq C \inf_{c\in\C}\Big( \|v-c\|_{H^1(\Omega)}^2 + \| V - v \|_{L^2(e)}^2 \Big) \nonumber \\[1mm]
& \leq C \Big( \| \nabla v \|_{L^2(\Omega)}^2 +  \| V - v \|_{L^2(e)}^2 \Big),
\end{align}
where the second step is a consequence of the Poincar\'e inequality. The assertion now follows by combining \eqref{eq:almostcoer} with the estimate
\begin{align*}
{\rm Re}  \Big( B \big((v,V), (v,V) \big) \Big) & \geq \, \varsigma_- \int_\Omega |\nabla v |^2 \,{\rm d}x + \zeta_- \int_{e} |V-v|^2\,{\rm d}S 
\end{align*}
that is induced by \eqref{eq:sigma}, \eqref{eq:zeta} and \eqref{eq:trueelec}.
\end{proof}

\begin{theorem}
\label{thm:existence}
Under the assumptions \eqref{eq:sigma} and \eqref{eq:zeta}, the problem \eqref{eq:cemeqs} has a unique solution $(u,U) \in \mathcal{H}^1$ for any 
current pattern $I \in \C_\diamond^M$. Moreover,
$$
\| (u,U) \|_{\mathcal{H}^1} \leq C \| I \|_2,
$$
where $C= C(\Omega,E,\sigma,\zeta)>0$ is independent of $I$.
\end{theorem}

\begin{proof}
Since $I \in \C_\diamond^M$ has vanishing mean, the map
$$
\mathcal{H}^1 \ni (v,V) \mapsto I \cdot \overline{V} \in \C
$$
is well-defined (and antilinear). Moreover, since
$$
\big| I \cdot \overline{V} \big| \leq \|I \|_2 \inf_{c \in \C} \|V - c {\bf 1}\|_2 \leq \|I\|_2 \, \| (v,V) \|_{\mathcal{H}^1},
$$
the claim follows by applying Lemma~\ref{lemma:coerc} and the Lax--Milgram theorem to the variational formulation \eqref{eq:weak}.
\end{proof}

If one resorts to the standard formulation of the CEM and replaces the second condition of \eqref{eq:cemeqs} by \eqref{eq:alternate} with contact impedances that are bounded away from infinity, the highest Sobolev regularity that the interior potential can, in general, exhibit is $u \in H^{2-\epsilon}(\Omega)$ due to the abrupt change of the boundary condition from Robin to homogeneous Neumann at $\partial E$ (see,~e.g.,~\cite{Costabel96}). With this in mind, an intriguing property of \eqref{eq:cemeqs} is that the smoothness of the interior electromagnetic potential is directly controlled by the smoothness of the contact admittance $\zeta: \partial \Omega \to \C$, assuming that the boundary $\partial \Omega$ and the admittivity $\sigma$ are smooth enough. We start with two simple lemmas:

\begin{lemma}
\label{lemma:multipli}
Assume that $\partial \Omega$ is of class $C^\infty$, $\Gamma \subset \partial \Omega$ is open with a Lipschitz boundary or $\Gamma = \partial \Omega$, and $\eta \in H^s(\Gamma)$ for some $s > (n-1)/2$. Then the multiplication operator
$$
\mathcal{M}_{\eta}: v \mapsto \eta v, \qquad H^r(\Gamma) \to H^r(\Gamma),
$$
is bounded for any $-s < r \leq s$. More precisely,
$$
\| \eta v \|_{H^r(\Gamma)} \leq C \| \eta \|_{H^s(\Gamma)} \| v \|_{H^r(\Gamma)}, \qquad -s < r \leq s,
$$
where $C = C(s,r,\Gamma) > 0$ is independent of $\eta$ and $v$.
\end{lemma}

\begin{proof}
The claim follows from,~e.g.,~\cite[p.~190, Theorem~1, (i)]{Runst96} by choosing $s_1 = r$, $s_2 = s$, $p=q=q_1=q_2=2$; see also \cite[Propositions on p.~14 and p.~150]{Runst96}.
\end{proof}

Notice that the assumption $s > (n-1)/2$, which recurs many times in the following, ensures that $H^s(\partial \Omega)$ is continuously embedded in the Banach space of continuous functions $\mathcal{C}(\partial \Omega)$ by virtue of the Sobolev embedding theorem.

\begin{lemma}
\label{lemma:dummy}
Assume that $\partial \Omega$ is of class $\mathcal{C}^\infty$. Then, for any $(w,W) \in \mathcal{H}^{r+1/2}$ with $r > 0$,
$$
\| W - w \|_{H^{r}(E)} \leq C \| (w, W) \|_{\mathcal{H}^{r+1/2}},
$$
where $C = C(r,E,\Omega) >0$.
\end{lemma}

\begin{proof}
We may estimate as follows:
\begin{align*}
\| W - w \|_{H^{r}(E)} &\leq \inf_{c \in \C} \Big( \| W - c\|_{H^{r}(E)} + \| c - w \|_{H^{r}(E)} \Big) \\ 
&\leq 2 \inf_{c \in \C} \Big( \| w - c\|_{H^{r}(\partial \Omega)}^2 + \| W - c\|_{H^{r}(E)}^2 \Big)^{1/2} \\[1mm]
&\leq C(r,E, \Omega) \| (w, W) \|_{\mathcal{H}^{r+1/2}},
\end{align*}
where the last step follows from the trace theorem and the equivalence of norms on a finite-dimensional space.
\end{proof}

The following theorem is the main result of this section. Take note that one could also prove a version where the required regularity of $\sigma$ and $\partial \Omega$ depends on the smoothness of $\zeta$,~i.e.~on $s$, but such a generalization would demand extra work without adding anything to the intuitive contents of the result. 
To be more precise, the proof would be analogous, but one would have to keep carefully track of how much regularity is required of $\partial \Omega$ and $\sigma$ to guarantee a certain smoothness for~$u$.

\begin{theorem}
\label{thm:smoothness}
Assume that \eqref{eq:sigma} and \eqref{eq:zeta} hold, $\sigma \in \mathcal{C}^\infty(\overline{\Omega}, \C^{n \times n})$  and  $\partial \Omega$ is of class $\mathcal{C}^\infty$. If furthermore $\zeta \in H^s(\partial \Omega)$ for some $s > (n-1)/2$, then the solution to \eqref{eq:weak} satisfies
\begin{equation}
\label{scont}
\| (u,U) \|_{\mathcal{H}^{s+3/2}} \leq C  \| I \|_2.
\end{equation}
where $C = C(\Omega, E, \sigma, \zeta, s) > 0$ is independent of $I \in \C_\diamond^M$.
\end{theorem}

\begin{proof}
To begin with, consider the Neumann boundary value problem
\begin{equation}
\label{eq:Neumann}
\nabla \cdot (\sigma \nabla v) = 0 \quad {\rm in} \ \Omega, \qquad \nu \cdot \sigma \nabla v = f \quad {\rm on} \ \partial \Omega 
\end{equation}
for an arbitrary $f$ in the mean-free Sobolev space
$$
H^r_\diamond(\partial \Omega) = \{g \in H^r(\partial \Omega) \ | \ \langle g, 1 \rangle_{\partial \Omega} = 0 \}, \qquad r \in \R,
$$
where $\langle \, \cdot \, , \, \cdot\, \rangle_{\partial \Omega}: H^{r}(\partial \Omega) \times H^{-r}(\partial \Omega) \to \C$ denotes the bilinear dual evaluation between Sobolev spaces on $\partial \Omega$.
Due to the infinite smoothness of $\sigma$ and $\partial \Omega$, for any $r \in \R$ the problem \eqref{eq:Neumann} has a unique solution $v \in H^{r+3/2}(\Omega) / \C$ satisfying~\cite{Lions72} 
\begin{equation}
\label{eq:contneum}
\| v \|_{H^{r+3/2}(\Omega) / \C} \leq C \| f \|_{H^r(\partial \Omega)}. 
\end{equation}
In particular, the Neumann-to-Dirichlet map
\begin{equation}
\label{eq:NtoD}
\Lambda_{\sigma}: f \mapsto v|_{\partial \Omega}, \quad H^r_\diamond(\partial \Omega) \to H^{r+1}(\partial \Omega) / \C,
\end{equation}
is bounded for any $r \in \R$ by virtue of trace theorems for those elements of $H^{r+3/2}(\Omega)$ for which the range of the differential operator $\nabla \cdot \sigma \nabla (\, \cdot \, )$ is a subspace of $L^2(\Omega)$~\cite{Lions72}. To complete this introductory part of the proof, let
$$
\mathcal{E}: {\rm span} \{ \chi_1, \dots, \chi_M\} \to H^s(\partial \Omega)
$$
be a linear and bounded extension operator (with respect to any given norm for the finite-dimensional space ${\rm span} \{ \chi_1, \dots, \chi_M\}$),~i.e.,~such that $(\mathcal{E}\eta) |_E = \eta|_E$ for all $\eta \in {\rm span} \{ \chi_1, \dots, \chi_M\}$. 

Let $(u,U) \in \mathcal{H}^1$ be the solution to \eqref{eq:cemeqs} for some $I \in \C_\diamond^M$.
The second equation of \eqref{eq:cemeqs}
immediately leads to the preliminary estimate
\begin{equation}
\label{eq:preli}
\| \nu \cdot \sigma \nabla u \|_{L^2(\partial \Omega)}
 \leq C \| \zeta \|_{H^s(E)} \| U - u \|_{L^2(E)} \leq C \| (u,U) \|_{\mathcal{H}^1}\leq C  \|I \|_2
\end{equation}
by virtue of Lemmas~\ref{lemma:multipli} and \ref{lemma:dummy}, and Theorem~\ref{thm:existence}.

Set $r = \min \{1,s\}$ and denote by $(\tilde{u}, \tilde{U})$ the particular representative of the equivalence class $(u,U) \in \mathcal{H}^1$ for which $\tilde{u}|_{\partial \Omega} \in H^{1/2}_\diamond(\partial \Omega)$.
Since $\zeta$ is supported on $\overline{E}$ by assumption~\eqref{eq:zeta}, the second condition of \eqref{eq:cemeqs} can be rewritten in an alternative form
$$
\nu \cdot \sigma \nabla u = \nu \cdot \sigma \nabla \tilde{u} =  \zeta ( \mathcal{E} \tilde{U}  - \tilde{u}) \qquad {\rm on} \ \partial \Omega.
$$
Because of Lemma~\ref{lemma:multipli}, the vanishing mean of $\tilde{u}$, the boundedness of the operators $\Lambda_\sigma$ and 
$\mathcal{E}$, and \eqref{eq:preli}, we thus have
\begin{align}
\label{eq:normalnorm}
\| \nu \cdot \sigma \nabla u \|_{H^{r}(\partial \Omega)}  
& \leq C \| \mathcal{E} \tilde{U} -  \tilde{u} \|_{H^{r}(\partial \Omega)} 
\leq C \Big( \| \mathcal{E} \tilde{U} \|_{H^{r}(\partial \Omega)} + \| \tilde{u}  \|_{H^{r}(\partial \Omega)/\C} \Big) \nonumber \\[0.5mm]
& \leq C \big( \| \tilde{U} \|_2 +  \| \nu \cdot \sigma \nabla u \|_{H^{r-1}(\partial \Omega)} \big) \leq C \big( \| \tilde{U} \|_2 +  \| I \|_2 \big).
\end{align}
Moreover, since $\zeta \in H^{s}(\partial \Omega) \subset \mathcal{C}(\partial \Omega)$ is nonzero on some set of nonzero measure on each $E_m$, $m=1, \dots, M$, due to~\eqref{eq:zeta}, the second equation of \eqref{eq:cemeqs} leads to
\begin{align*}
\|\tilde{U}\|_2 &\leq C \int_{E} |\zeta \tilde{U}| \, {\rm d} S 
\leq C \Big( \| \zeta \|_{L^\infty(\partial \Omega)} \int_{E} |\tilde{u}| \, {\rm d} S + \int_{E} | \nu \cdot \sigma \nabla u | \, {\rm d} S \Big) \\
&\leq C \big( \| \tilde{u} \|_{L^2(\partial \Omega)} + \|  \nu \cdot \sigma \nabla u \|_{L^2(\partial \Omega)} \Big) \leq C \|I \|_2,
\end{align*}
where $C>0$ does not depend on $I$ (or $\tilde{U}$) and the last step is a consequence of \eqref{eq:preli} and the boundedness of $\Lambda_\sigma: L^2_\diamond(\partial \Omega) \to L^2(\partial \Omega)/\C$. Combining the previous estimate with \eqref{eq:contneum} and \eqref{eq:normalnorm}, results in
\begin{equation}
\label{eq:preliplus}
\| u \|_{H^{r+3/2}(\Omega)/ \C} \leq C \| \nu \cdot \sigma \nabla u \|_{H^{r}(\partial \Omega)} \leq C \|I \|_2.
\end{equation}
In particular, it is straightforward to deduce that
\begin{equation}
\label{eq:end}
\| (u,U) \|_{\mathcal{H}^{r+3/2}} \leq C (\| u \|_{H^{r+3/2}(\Omega)/ \C} + \| (u,U) \|_{\mathcal{H}^{1}}) \leq C \|I\|_2, 
\end{equation}
which proves the claim if $s \leq 1$.

If $s > 1$, one can repeat the above argument with $r = \min\{2,s \}$, using the previous \eqref{eq:preliplus} in place of \eqref{eq:preli} in the new \eqref{eq:normalnorm}, to end up once again with \eqref{eq:end}, which this time around corresponds to \eqref{scont} if $s \leq 2$. The complete assertion follows by inductively reiterating this argument.
\end{proof}

\begin{remark}
\label{remark}
The assumption on the smoothness of the boundary $\partial \Omega$ can be relaxed in many ways. For example, if $n=2$, the domain is a square and the corners are well separated from the electrodes, one can locally straighten the corners by introducing a suitable conformal map ($z \mapsto z^2$). The crucial property is the vanishing Neumann boundary condition that is preserved when the domain is deformed. 
\end{remark}

\section{Shape derivatives}
\label{sec:Frechet}

As an example of a setting where the higher regularity of the smoothened CEM may be useful, we consider the Fr\'echet derivative of electrode measurements with respect to the shape of $\Omega$. Such a derivative is needed in absolute EIT if the body shape is not accurately known (cf.,~e.g.,~medical imaging) and must be reconstructed at the same time as the admittivity to avoid severe artifacts~\cite{Darde13a,Darde13b,Nissinen11a}. For simplicity and to be able to use the results in \cite{Darde12,Darde13a} without further generalizations, we assume throughout this section that $\partial \Omega$ and $\partial E$ are smooth.

The measurement, or current-to-voltage map of the CEM is usually defined via
\begin{equation}
\label{eq:measmat}
R: I \mapsto U, \quad \C^M_\diamond \to \C^M / \C ,
\end{equation} 
where $U$ is the second component of the solution to \eqref{eq:cemeqs} 
and $\C^M / \C = \{ \{ V + c \mathbf{1} \, | \, c \in \C\} \, | \, V \in \C^M \}$.  Due to a certain symmetry of \eqref{eq:weak}, $R$ can be represented as a real-symmetric, complex $(M-1)\times(M-1)$ matrix in terms of any orthonormal basis for $\C^M_\diamond \sim \C^M/\C$. 
However, assuming that the admittivity $\sigma$ is defined in some neighborhood of $\Omega$, $R$ can be interpreted as a function of two variables,
\[
R: (I, h) \mapsto U(I,h), \quad \C^M_\diamond \times \mathcal{B}_d
\to \C^M / \C,
\]
where $\mathcal{B}_d \subset [\mathcal{C}^1(\partial\Omega)]^n$ is an origin-centered open ball of radius $d>0$. Moreover, $(u(I,h),U(I,h))$ is the solution of \eqref{eq:cemeqs} when $\partial \Omega$ and $E_m$, $m=1, \dots, M$, are replaced by the perturbed versions
\begin{equation}
\label{eq:perturbed}
\partial \Omega^h = \{ x + h(x) \, | \, x \in \partial \Omega \}, \qquad
E_m^h = \{ x + h(x) \, | \, x \in E_m \},
\end{equation}
respectively. The surface admittance $\zeta$ is assumed to stretch accordingly,~i.e.,~the electrode contacts on $\partial \Omega^h$ are characterized by the function 
$$
\zeta^h: \partial \Omega^h \ni x+h(x) \mapsto \zeta(x) \in \C.
$$ 
If $d > 0$ is chosen small enough, the above definitions are unambiguous in the sense that $\partial \Omega^h$ defines a $\mathcal{C}^1$ boundary of a bounded domain $\Omega^h$, $\{E^h_m\}_{m=1}^M$ is a proper set of connected, well-separated electrodes and $\zeta^h$ defines an element of $L^\infty(\partial \Omega^h)$ satisfying \eqref{eq:zeta} with respect to the perturbed boundary and electrodes~\cite{Darde13a}. In the following, we will implicitly assume that $d>0$ is chosen small enough in this sense.

For the traditional CEM, the elliptic boundary value problem defining the Fr\'echet derivative of $R: \C^M_\diamond \times \mathcal{B}_d \to \C^M / \C$ with respect to its second variable falls outside the $H^1$-based variational theory; to be more precise, the solution of the shape derivative problem belongs to $\mathcal{H}^{1-\epsilon}$ for any $\epsilon>0$ \cite[Theorem~3.3]{Darde13a}. However, for the smoothened version \eqref{eq:cemeqs}, the problem defining the shape derivative admits a $\mathcal{H}^1$-based variational formulation with the same left-hand side as in \eqref{eq:weak}, if $\zeta$ is regular enough. 

Assume that $\zeta \in H^s(\partial \Omega)$ with some $s > (n-1)/2$ and let $(u,U) \in \mathcal{H}^{s+3/2}$ be the solution of \eqref{eq:cemeqs} for some $I \in \C_\diamond^M$. Consider the problem of finding $(u'[h], U'[h]) \in \mathcal{H}^{1}$ such that
\begin{equation}
\label{eq:derivari}
B\big((u'[h],U'[h]), (v,V)\big) = F[h](v,V) \qquad {\rm for} \ {\rm all} \ (v,V) \in \mathcal{H}^1,
\end{equation}
where the antilinear functional $F[h]: \mathcal{H}^1 \to \C$ is defined by
\begin{align}
\label{eq:aFFa}
F[h](v,V) &= \int_{\partial \Omega} h_\nu \zeta \Big(\frac{\partial u}{\partial \nu} - \kappa (U-u)\Big) (\overline{V} - \overline{v}) \dS \nonumber \\[1mm]
 & \quad + \big\langle h_\tau \cdot {\rm Grad}\, \zeta,  (U-u) (\overline{V} - \overline{v}) \big\rangle_{E} \\[1mm]
&\quad - \big\langle h_\nu(\sigma \nabla u |_{\partial \Omega})_\tau, {\rm Grad} (\overline{v}|_{\partial \Omega}) \big \rangle_{\partial \Omega}. \nonumber
\end{align}
Here, ${\rm Grad}$ denotes the surface gradient \cite{Colton98}, $\kappa: \partial \Omega \to \R$ is the sum of principal curvatures on $\partial \Omega$, and $h_\nu = h \cdot \nu$ and $h_\tau = h - h_\nu \nu$ are the normal (scalar) and tangential (vector) components of $h: \partial \Omega \to \R^{n}$, respectively. Moreover,  $\langle \, \cdot \, , \, \cdot\, \rangle_{E}: H^{-r}(E) \times H^{r}(E) \to \C$,  $|r|<1/2$,  denotes the bilinear dual evaluation between Sobolev spaces on the electrodes (see,~e.g.,~\cite{Lions72}). Finally, note that $F[h]$ depends linearly on $h$, and so the same conclusion also holds for the solution $(u'[h],U'[h]) \in \mathcal{H}^1$, if it uniquely exists.

\begin{lemma}
\label{lemma:derex}
Assume that \eqref{eq:sigma} and \eqref{eq:zeta} hold, $\sigma \in \mathcal{C}^\infty(\overline{\Omega}, \C^{n \times n})$, and $\partial \Omega$ and $\partial E$ are of class $\mathcal{C}^\infty$. If furthermore $\zeta \in H^s(\partial \Omega)$ for some $s > (n-1)/2$, then the variational problem \eqref{eq:derivari} has a unique solution $(u'[h], U'[h]) \in \mathcal{H}^{1}$ for any $h \in [\mathcal{C}^1(\partial \Omega)]^n$ and $I \in \C_\diamond^M$, satisfying
$$
\| (u'[h], U'[h]) \|_{\mathcal{H}^1} \leq C \| I \|_2 \|h\|_{\mathcal{C}^1(\partial \Omega)},
$$
where $C> 0$ is independent of $h$ and $I$.
\end{lemma}
\begin{proof}
By virtue of Lemma~\ref{lemma:coerc} and the Lax--Milgram theorem, it is enough to prove that 
$$
| F[h] (v,V)| \leq C \| I \|_2  \|h\|_{\mathcal{C}^1(\partial \Omega)} \|(v,V) \|_{\mathcal{H}^1}
$$
for any $h \in [\mathcal{C}^1(\partial \Omega)]^n$ and $I \in \C_\diamond^M$. By resorting to the triangle inequality, we may handle the terms on the right-hand side of \eqref{eq:aFFa} one by one.

According to the Cauchy--Schwarz inequality and Lemma~\ref{lemma:dummy},
\begin{align*}
\Big| \int_{\partial \Omega} h_\nu \zeta \Big(\frac{\partial u}{\partial \nu} - \kappa (U-u)\Big) (\overline{V} - \overline{v}) \dS  \Big| 
\leq 
\Big\| h_\nu \zeta \Big(\frac{\partial u}{\partial \nu} - \kappa (U-u)\Big) \Big \|_{L^2(E)} \| (V,v) \|_{\mathcal{H}^1}.
\end{align*}
Applying the triangle inequality and Lemma~\ref{lemma:multipli} (with $\eta = \zeta|_E \in H^s(E)$, $\eta = h_\nu|_E \in \mathcal{C}^1(\overline{E}) \hookrightarrow H^1(E)$ and $\eta = \kappa|_E \in \mathcal{C}^{\infty}(\overline{E})$) to the first factor on the right-hand side  gives
\begin{align*}
\Big\| h_\nu \zeta \Big(\frac{\partial u}{\partial \nu} - \kappa (U-u)\Big) \Big \|_{L^2(E)} & \leq C \|h\|_{\mathcal{C}^1(\partial \Omega)} \Big( \Big\| \frac{\partial u}{\partial \nu} \Big\|_{L^2(E)} 
+ \| U - u \|_{L^{2}(E)} \Big)  \\[1mm]
& \leq C \|h\|_{\mathcal{C}^1(\partial \Omega)} \| (u, U) \|_{ \mathcal{H}^{2}} \leq C \|I\|_2 \|h\|_{\mathcal{C}^1(\partial \Omega)} ,
\end{align*}
where the last two steps follow from the trace theorem, Lemma~\ref{lemma:dummy} and Theorem~\ref{thm:smoothness}. This takes care of the first term on the right-hand side of \eqref{eq:aFFa}.

Set $\delta = \min\{1/2, (s - 1/2)/2\}$, so that $0 < \delta \leq 1/2$ and $1/2 + \delta < s$. 
Lemmas~\ref{lemma:multipli} and \ref{lemma:dummy} yield
\begin{align*}
\big\| (U-u)(\overline{V} - \overline{v}) \big\|_{H^{1/2-\delta}(E)} & \leq 
C  \| U-u\|_{H^{s}(E)} \| V-v\|_{H^{1/2-\delta}(E)} \\[1mm]
&\leq C \| (u,U)\|_{\mathcal{H}^{s+1/2}}\| (v,V)\|_{\mathcal{H}^{1}} \leq C \|I\|_2 \| (v,V)\|_{\mathcal{H}^{1}},
\end{align*}
where the last step corresponds to a weaker version of Theorem~\ref{thm:smoothness}.
Moreover, by the continuity of ${\rm Grad}: H^s(E) \supset H^{1/2+\delta}(E) \to [H^{-1/2+\delta}(E)]^{n}$ (cf.~\cite[p.~85, Prop.~12.1]{Lions72}) and Lemma~\ref{lemma:multipli} with $\eta = [h_\tau|_E]_j \in \mathcal{C}^1(\overline{E}) \hookrightarrow H^1(E)$, $j=1,\dots, n-1$, 
$$
\| h_\tau \cdot {\rm Grad}\, \zeta \|_{H^{-1/2 + \delta}(E)}
\leq C  \| \zeta \|_{H^{s}(E)} \| h \|_{\mathcal{C}^1(\partial \Omega)} \leq C  \| h \|_{\mathcal{C}^1(\partial \Omega)}.
$$
In consequence,
\begin{align*}
\big| \big\langle h_\tau \cdot {\rm Grad}\, \zeta &,  (U-u) (\overline{V} - \overline{v}) \big\rangle_{E} \big| \leq C \|I\|_2  \| h \|_{\mathcal{C}^1(\partial \Omega)}\| (v,V)\|_{\mathcal{H}^{1}}, 
\end{align*}
which handles the second term on the right-hand side of \eqref{eq:aFFa}.

Finally, by using the continuity of ${\rm Grad}: H^{1/2}(\partial \Omega)/\C \to [H^{-1/2}(\partial \Omega)]^{n}$, Lemma \ref{lemma:multipli} with $\eta = h_\nu \in \mathcal{C}^1(\partial \Omega) \hookrightarrow H^1(\partial \Omega)$ and the trace theorem,
\begin{align*}
\big | \big \langle h_\nu(\sigma \nabla u |_{\partial \Omega})_\tau, {\rm Grad} (\overline{v}|_{\partial \Omega}) \big \rangle_{\partial \Omega} \big|
&\leq C \big \|h_\nu(\sigma \nabla u |_{\partial \Omega})_\tau \big\|_{H^{1/2}(\partial \Omega)} \| v \|_{H^{1/2}(\partial \Omega)/\C} \\[1mm]
& \leq C \| h \|_{\mathcal{C}^1(\partial \Omega)} \| u \|_{H^2(\Omega)/\C} \| v \|_{H^{1}(\Omega)/\C} \\[1mm]
& \leq C \|I\|_2  \| h \|_{\mathcal{C}^1(\partial \Omega)}  \| (v,V) \|_{\mathcal{H}^{1}},
\end{align*}
where the final step follows from Theorem~\ref{thm:smoothness}. This completes the proof.
\end{proof}

As hinted above, the problem \eqref{eq:derivari} defines the Fr\'echet derivative of the measurement map $R$ with respect to its second variable, that is, with respect to the shape of $\Omega$.

\begin{theorem}
\label{thm:reunaderivaatta}
Assume that \eqref{eq:sigma} and \eqref{eq:zeta} hold, $\sigma \in \mathcal{C}^\infty(\overline{\Omega}, \C^{n \times n})$, and $\partial \Omega$ and $\partial E$ are of class $\mathcal{C}^\infty$. If furthermore $\zeta \in H^s(\partial \Omega)$ for some $s > (n-1)/2$, then $R:  \C^M_\diamond \times \mathcal{B}_d \to \C^M / \C$ is Fr\'echet differentiable with respect to its second variable at the origin, with the Fr\'echet derivative given by the linear and bounded map
$$
[C^1(\partial\Omega)]^n \ni h \mapsto U'[h] \in \C^M / \C,
$$
where $U'[h]$ is the second component of the solution to \eqref{eq:derivari}.
\end{theorem}

\begin{proof}
The assertion follows by slightly modifying (or simplifying) the proofs in \cite{Darde12, Darde13a} taking into account that now $(u,U) \in \mathcal{H}^2$ unlike in \cite{Darde12, Darde13a}.
\end{proof}

If $I, \tilde{I}\in\C^M_\diamond $ are electrode current patterns and $(u,U), (\tilde{u},\tilde{U})$ are the respective solutions of \eqref{eq:cemeqs}, then $U'[h]$ can be assembled using the relation
\begin{align}
\label{eq:sampling1}
U'[h]\cdot \tilde{I} &= \int_{\partial \Omega}h_\nu \zeta \Big( \frac{\partial u}{\partial\nu} - \kappa (U-u) \Big)(\tilde{U}-\tilde{u}) \, {\rm d}S \\[1mm]
\label{eq:sampling2}
& \quad + \int_{\partial \Omega} (h_\tau \cdot {\rm Grad} \, \zeta) (U-u)
(\tilde{U}-\tilde{u}) \,{\rm d}S \\[1mm]
\label{eq:sampling3}
& \quad -\int_{{\partial\Omega}} h_\nu\,(\sigma\nabla u)_\tau \cdot
(\nabla\tilde{u})_\tau \, {\rm d}S,
\end{align}
where it is assumed that $\zeta \in H^{\max\{1,s\}}(\partial \Omega)$ for some $s > (n-1)/2$.
This formula can be straightforwardly deduced from the variational problems \eqref{eq:weak} and \eqref{eq:derivari}: Use the complex conjugate of $(u'[h],U'[h])$ as the test element in \eqref{eq:weak} with $I = \tilde{I}$, then interpret the complex conjugate of $(\tilde{u},\tilde{U})$ as the test element in \eqref{eq:derivari}, and finally employ the higher smoothness of the pairs  $(u,U)$ and $(\tilde{u},\tilde{U})$ guaranteed by Theorem~\ref{thm:smoothness} to interpret the dual evaluations in \eqref{eq:aFFa} as regular integrals (cf.~\cite{Darde12,Darde13a}).
If $\sigma$ is scalar-valued (i.e.,~isotropic), the integral \eqref{eq:sampling1} can be rewritten as
\begin{equation*}
\int_{\partial \Omega} h_\nu \zeta \Big( \frac{\zeta}{\sigma} - \kappa \Big) (U-u) (\tilde{U}-\tilde{u}) \, {\rm d}S,
\end{equation*}
which reveals a symmetry between $(u,U)$ and $(\tilde{u},\tilde{U})$, 
that is, the roles of $(u,U)$ and $(\tilde{u},\tilde{U})$ can be reversed on the right-hand side of \eqref{eq:sampling1}--\eqref{eq:sampling3} without altering its value.

To complete this section, take note that the sampling formula \eqref{eq:sampling1}--\eqref{eq:sampling3} also holds if $\zeta$ is less regular, in which case \eqref{eq:sampling2} just needs to be interpreted as an appropriate dual evaluation. In particular, if the second equation of \eqref{eq:cemeqs} is replaced by the conditions \eqref{eq:alternate} of the traditional CEM, the term $h_\tau \cdot {\rm Grad} \, \zeta$ becomes a weighted delta distribution supported on $\partial E$ and \eqref{eq:sampling2} turns into an integral over $\partial E$~\cite[Corollary~3.4]{Darde13a}.

\section{Numerical experiments}
\label{sec:numerics}

The numerical implementation of the proposed smoothened CEM for EIT is a direct generalization of the traditional CEM implementation presented in, e.g.,~\cite{Kaipio00}. Merely, the nonhomogeneous admittance $\zeta$ has to be incorporated into the associated integrals. As discussed above,
the same also holds true for the numerical shape derivatives, with the exception of \eqref{eq:sampling2} that reduces to an integral over the electrode boundaries in the traditional CEM. 

In this section, we restrict our attention to isotropic, real-valued material parameters and currents. In particular, $\sigma$ denotes the electrical conductivity and $\zeta$ is the contact conductance. 
This is the most commonly considered setting in practical EIT: For anisotropic conductivities, the inverse problem of EIT suffers from an inherent nonuniqueness \cite{Sylvester90} and, on the other hand, the temporal frequencies employed by EIT devices are often so low that one can ignore the capacitive effects, which leads to a real-valued $\sigma$ (see \cite{Vauhkonen97} and references therein). Be that as it may, some of the conclusions drawn in what follows may not apply as such to anisotropic or/and complex conductivities. 
We only work in two dimensions for the sake of computational simplicity. 

Our numerical experiments are divided into four parts: 
First, we investigate the discrepancy between electrode measurements predicted by the traditional CEM and by a particular version of the smoothened CEM (see Figure~\ref{fig:setup}). Next, we study the convergence of FEM for the two forward models as well as for the shape derivative integrals \eqref{eq:sampling1}--\eqref{eq:sampling3}. The third test compares predictions of the two models to experimental EIT data from a homogeneous water tank. 
Finally, using experimental data corresponding to an insulating inclusion in the aforementioned water tank, we demonstrate that the quality of a reconstruction produced by a Bayesian algorithm, which aims at finding a {\em maximum a posteriori} (MAP) estimate for the conductivity and the contact conductance parameters, is independent of the choice between the two models.

\subsection{Model differences}
\label{sec:modelnum}

In order to study the properties of the smoothened CEM, we set up a test geometry on the unit square $\Omega=(0,1)^2$ with $M=8$ electrodes as shown on the left in Figure~\ref{fig:setup}. 
A square is our choice of test domain because it allows an exact representation as a union of triangles, and thus the discretization of the domain boundary does not induce any extra error in the considered FEM solutions; cf.~Remark~\ref{remark}.
In what follows, the numerical solutions for the elliptic problem \eqref{eq:weak} are computed by FEM on uniform meshes similar to the one shown in Figure~\ref{fig:setup}.
The number of nodes ranges from $9^2$ to $2049^2$ and in each case the electrodes cover an even number of element edges.
Unless otherwise stated, the discretization is piecewise linear so that the number of nodes equals the number of degrees of freedom.
The width of a boundary edge element is denoted by $h>0$.

\begin{figure}
\center{
{\includegraphics[height=1.8in]{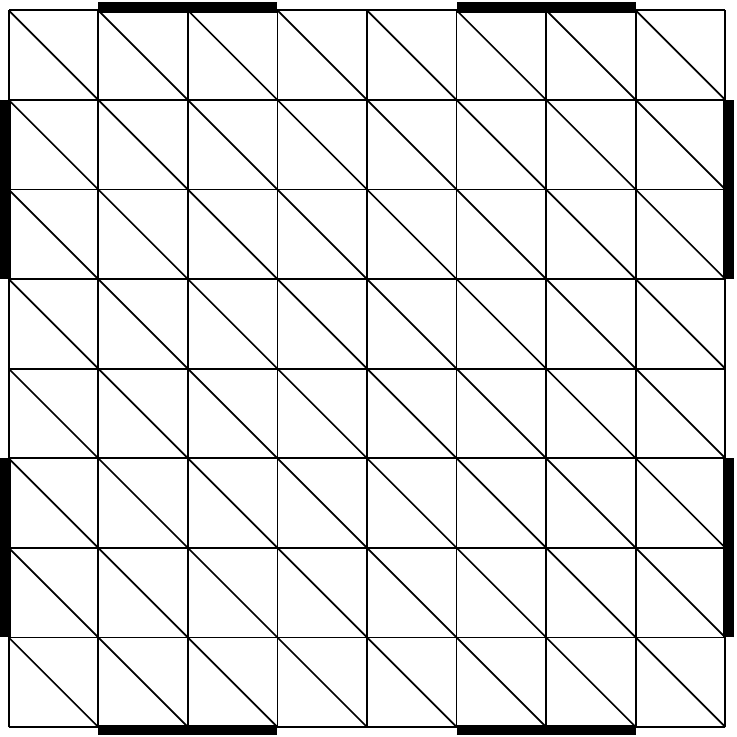}}
\qquad
{\includegraphics[height=1.8in]{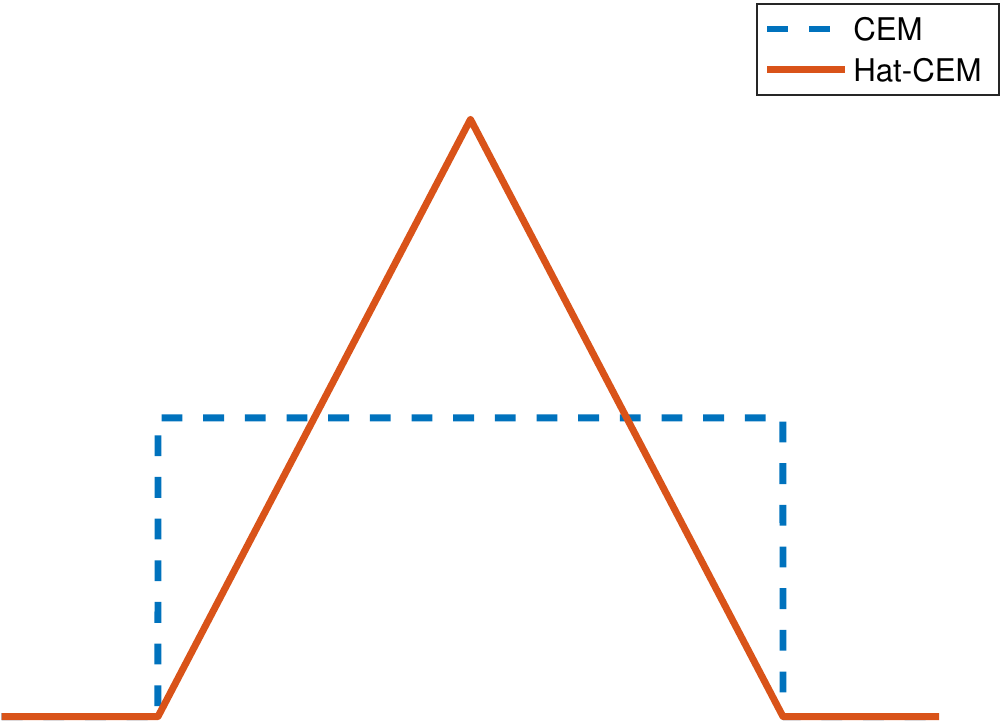}}\\
}
\caption{Left:\ Discretized unit square and eight electrodes depicted by thick line segments. Right:\ Two conductance functions $\zeta$ at the proximity of one electrode. The dashed line corresponds to the traditional CEM and the solid line is a piecewise linear ``hat'' function with the same average conductance.}
\label{fig:setup}
\end{figure}

Recall that the standard CEM is recovered from our setting by defining $\zeta$ to be constant on each electrode and zero elsewhere, cf.\ \eqref{eq:alternate}.
Let us replace these box-shaped functions by hat functions with an equal area under their graphs as depicted on the right in Figure~\ref{fig:setup}.
Such a conductance function satisfies $\zeta \in H^{3/2-\epsilon}(\Gamma)$ for any $\epsilon>0$ and $\Gamma \subset \partial \Omega$ that does not contain any of the four corners. According to Theorem~\ref{thm:smoothness} and Remark~\ref{remark}, the corresponding solution $(u,U)$ thus belongs to $\mathcal{H}^{3-\epsilon}$ provided that the conductivity $\sigma$ is smooth enough.

Before considering
convergence of FEM or real measurements, 
a natural question to ask is how the electrode potentials produced by the new model differ from the standard CEM when both forward problems are solved numerically.
To this end, we choose a finite element mesh with $1025^2$ nodes and define the relative difference
\begin{equation}
\label{eq:du}
d_U(\sigma,\zeta_\text{el}) := \left( \sum_{m=1}^{M-1} \left\lVert U^{(m)} - U^{(m)}_\text{CEM} \right\rVert_2^2 \right)^{1/2} \Bigg/
       \left( \sum_{m=1}^{M-1} \left\lVert U^{(m)}_\text{CEM} \right\rVert_2^2 \right)^{1/2},
\end{equation}
where $U^{(m)}$ and $U^{(m)}_\text{CEM}$ denote the electrode potential vectors for the hat model and the standard model, respectively, corresponding to a common current pattern $I^{(m)} = \mathrm{e}_M - \mathrm{e}_m \in \R^M_\diamond$. Here, $\mathrm{e}_m$ denotes the $m$th Cartesian basis vector of $\R^M$.
The unit or the magnitude of the current are not relevant since they cancel out due to linearity.
It is assumed that the conductivity $\sigma>0$ is constant and that the contact conductance, characterized by the height of the box function (or the half-height of the hat function), is the same for each electrode.
We denote this conductance value by $\zeta_\text{el} > 0$.
Now the relative difference \eqref{eq:du} actually depends only on the ratio $\sigma/\zeta_\text{el} \in \R$, as can be easily deduced from \eqref{eq:cemeqs}.
The quantity $\sigma/\zeta_\text{el}$ has the unit of length (e.g., meters).

It turns out that the relative difference $d_U$ is around $9 \%$ for some values of $\sigma/\zeta_\text{el}$, but it clearly diminishes if the conductance value is either very high or very low.
This behavior is illustrated in the top left image of Figure~\ref{fig:potdiff}.
The next step is to study whether the discrepancy between the models can be kept smaller if the hat functions are scaled differently than in Figure~\ref{fig:setup}.
Given a ratio $\sigma/\zeta_\text{el}$ for the standard CEM, we thus try to find $\zeta_\text{el}'>0$, the half-height of $\zeta$ in our smoothened model, that minimizes the corresponding difference $d_U'(\sigma, \zeta_\text{el}, \zeta_\text{el}')$, which is computed as \eqref{eq:du} but now with this new scaling for the hat function.
For given values of $\sigma$ and $\zeta_\text{el}$, finding $\zeta_\text{el}'$ is a well-defined one-dimensional minimization problem.
The plot at the bottom in Figure~\ref{fig:potdiff} presents the dependence between $\zeta_\text{el}$ and $\zeta_\text{el}'$, whereas the top right image shows the smallest relative difference $d_U'$ as a function of $\sigma/\zeta_\text{el}$.
The difference attains its maximum value of about $5.8\cdot10^{-3}$ at $\sigma/\zeta_\text{el} \approx 50 \cdot 10^{-3} \,\mathrm{m}$ and clearly decreases when the ratio is changed to either direction. 
Some numerical instabilities occur when the ratio gets very low, that is, when the setting approaches the shunt model~\cite{Darde16}. 

\begin{figure}
\center{
{\includegraphics[scale=.5]{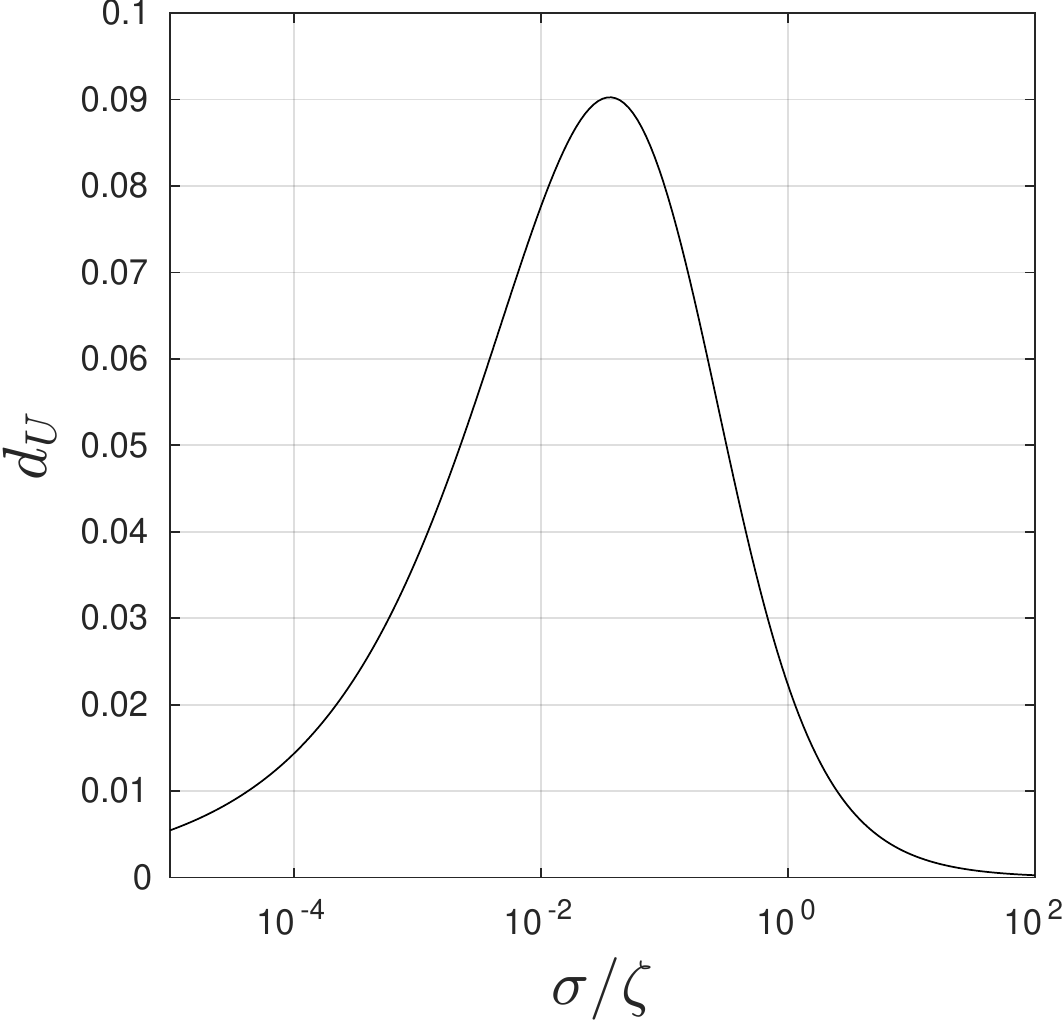}}
\qquad
{\includegraphics[scale=.5]{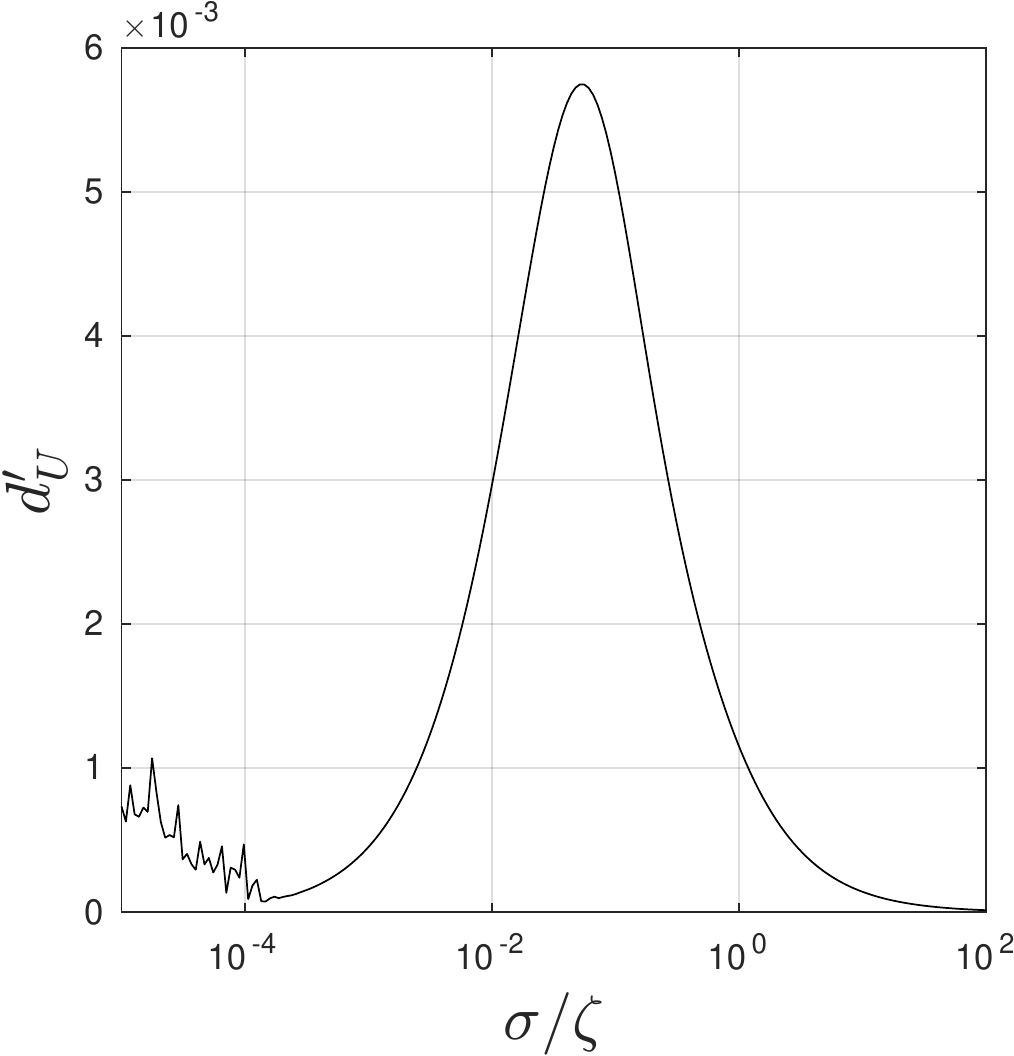}}
}
\includegraphics[scale=.5]{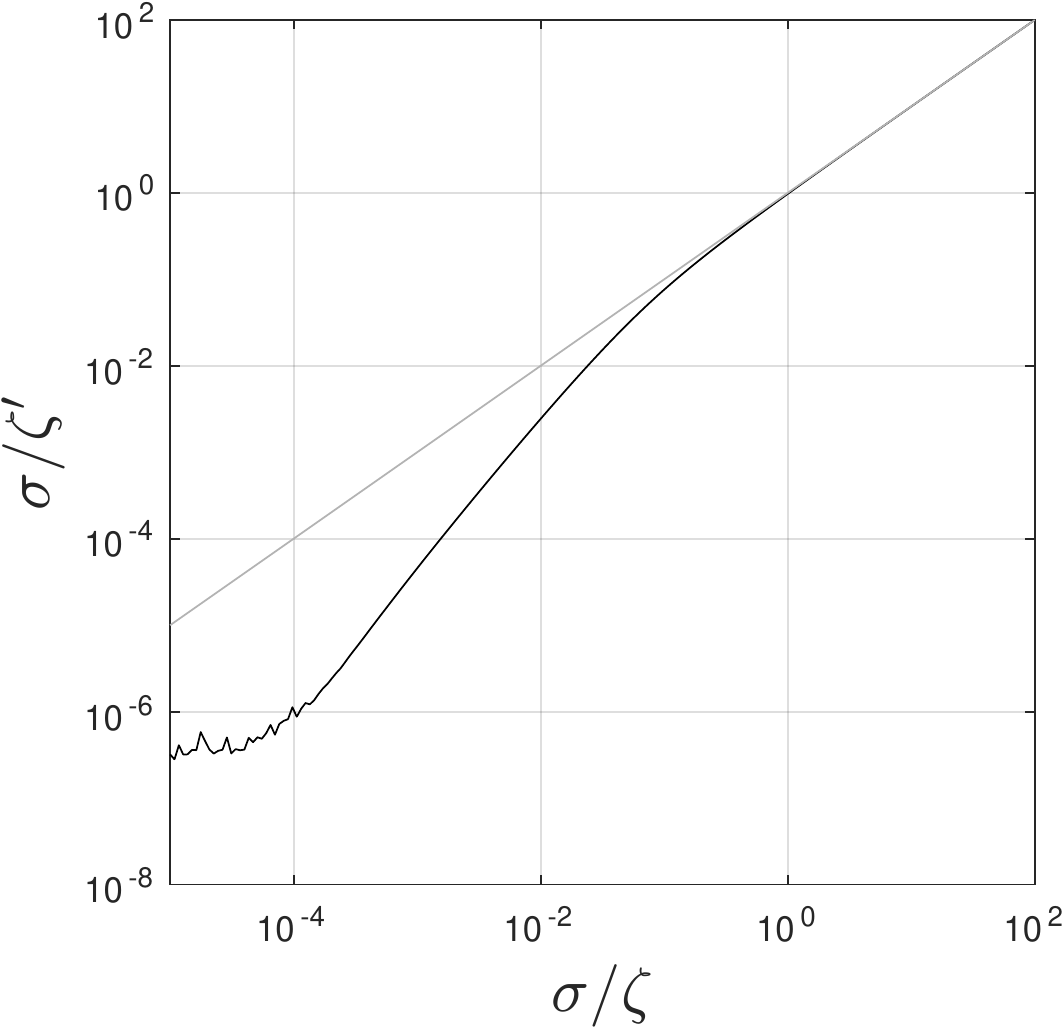}
\caption{Top left:\ Difference of the standard CEM and the proposed model with default scaling. Top right:\ Smallest possible difference obtained by optimized conductance scaling. Bottom:\ Dependence of the optimally scaled hat-conductance parameter $\zeta_\text{el}'$ on the standard CEM conductance $\zeta_\text{el}$. The line $\zeta_\text{el}' = \zeta_\text{el}$ is shown for comparison.}
\label{fig:potdiff}
\end{figure}

When dealing with real-world measurements, the scaling of hat functions is most likely not an issue since the values of the contact conductances are usually not interesting \emph{per se}, but they are merely estimated in order to obtain a more accurate reconstruction for the conductivity.
Moreover, there is no reason to expect that the standard CEM with piecewise continuous conductances would be a totally accurate measurement model either, which is why the quantity \eqref{eq:du} is not called an ``error'' but a ``difference''.

In \cite{Cheng89}, the ratio $\sigma/\zeta_\text{el} \approx 2.4\cdot10^{-3} \,\mathrm{m}$ was reported in several saline experiments.
Regarding EIT with tap water environment, the prior means used in \cite{Darde13b} correspond to $1.3 \cdot 10^{-3} \,\mathrm{m}$, whereas in \cite{Hyvonen17} the initial guess for the conductivity and the contact conductance results in $\sigma/\zeta_\text{el} \approx 0.7 \cdot 10^{-3} \,\mathrm{m}$.
In \cite{Hyvonen15}, the permissible ratio varies between $1.1 \cdot 10^{-5} \,\mathrm{m}$ and $1.1 \cdot 10^{-3} \,\mathrm{m}$.
All these experiments were performed with cylindrically symmetric water tanks with circumferences of about $1 \,\mathrm{m}$, i.e., roughly one fourth of the circumference of our unit square.
Thus, the corresponding points in Figure~\ref{fig:potdiff} can be found by multiplying these ratios by 4, assuming that our observations can be generalized to cases where the geometry and the number of electrodes differ from our test setup.
In any case, it seems that the lowest reported values for $\sigma/\zeta_\text{el}$ are well below the peak in the top right plot of Figure~\ref{fig:potdiff}, actually approaching the shunt model. 
However, one cannot exclude the possibility of encountering conductivity-conductance ratios that correspond to the largest mismatch between the two models because the contact conductances can vary significantly, e.g., in EIT imaging of concrete~\cite{Karhunen10}. 
On the other hand, typical measurement noise levels are clearly nonnegligible compared to any value on the modeling difference graph in the top right image of Figure~\ref{fig:potdiff}; cf.,~e.g.,~\cite{Kourunen09} where the performance of  the EIT unit used in Sections~\ref{sec:reco0} and \ref{sec:reco} is analyzed.

\subsection{Convergence of FEM}
\label{sec:conv}

Let us continue using the test setup of Figure~\ref{fig:setup} for both the standard CEM and for the hat function conductance model.
We study the convergence of finite element solutions toward ``exact'' reference solutions that are computed for both electrode models by using the finest mesh with $2049^2$ nodes.
In addition to the piecewise linear FEM basis functions, the quadratic Lagrange elements are used for comparison; 
for general information on properties and advantages of FEMs of different types and order, we refer to the textbook \cite{Larson13} and the references therein. 
Two pairs of conductivities and contact conductances are considered:\ the first one corresponds to $\sigma / \zeta_\text{el} \approx 50 \cdot 10^{-3} \,\mathrm{m}$, i.e., the peak in the top right plot in Figure~\ref{fig:potdiff}, whereas the second pair 
results in $\sigma / \zeta_\text{el} \approx 4 \cdot 10^{-3} \,\mathrm{m}$, which roughly corresponds to the values mentioned in~\cite{Darde13b}.
The conductance half-heights $\zeta_\text{el}'$ for the smoothened model are computed as in the bottom plot of Figure~\ref{fig:potdiff} and they are $\sigma/\zeta_\text{el}' \approx 30 \cdot 10^{-3} \,\mathrm{m}$ and $\sigma/\zeta_\text{el}' \approx 0.5 \cdot 10^{-3} \,\mathrm{m}$, respectively.
The relative errors are computed in the same way as in~\eqref{eq:du}, but now separately for each conductance model and comparing against the respective reference solution.

\begin{figure}
\center{
{\includegraphics[scale=.4]{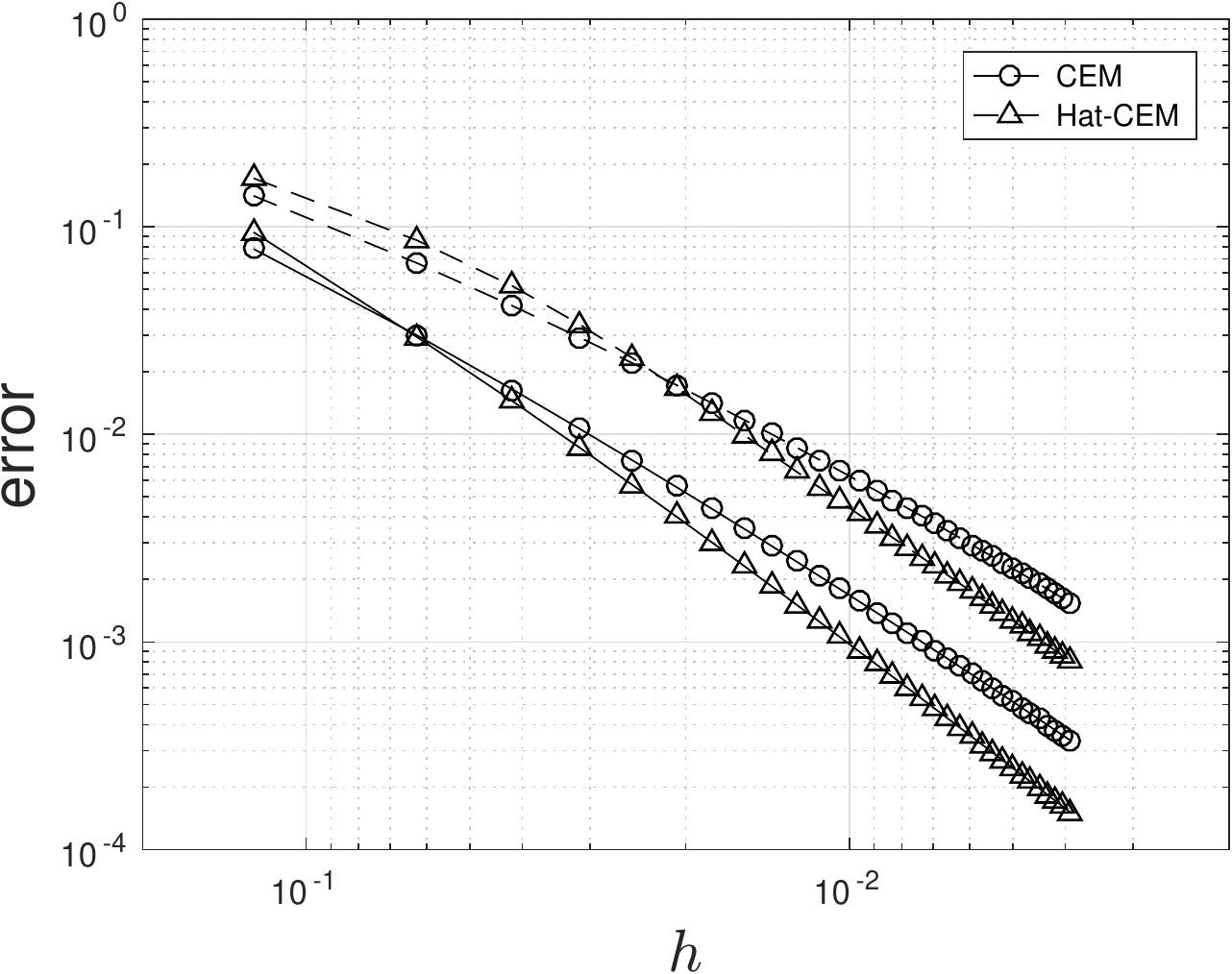}}
\qquad
{\includegraphics[scale=.4]{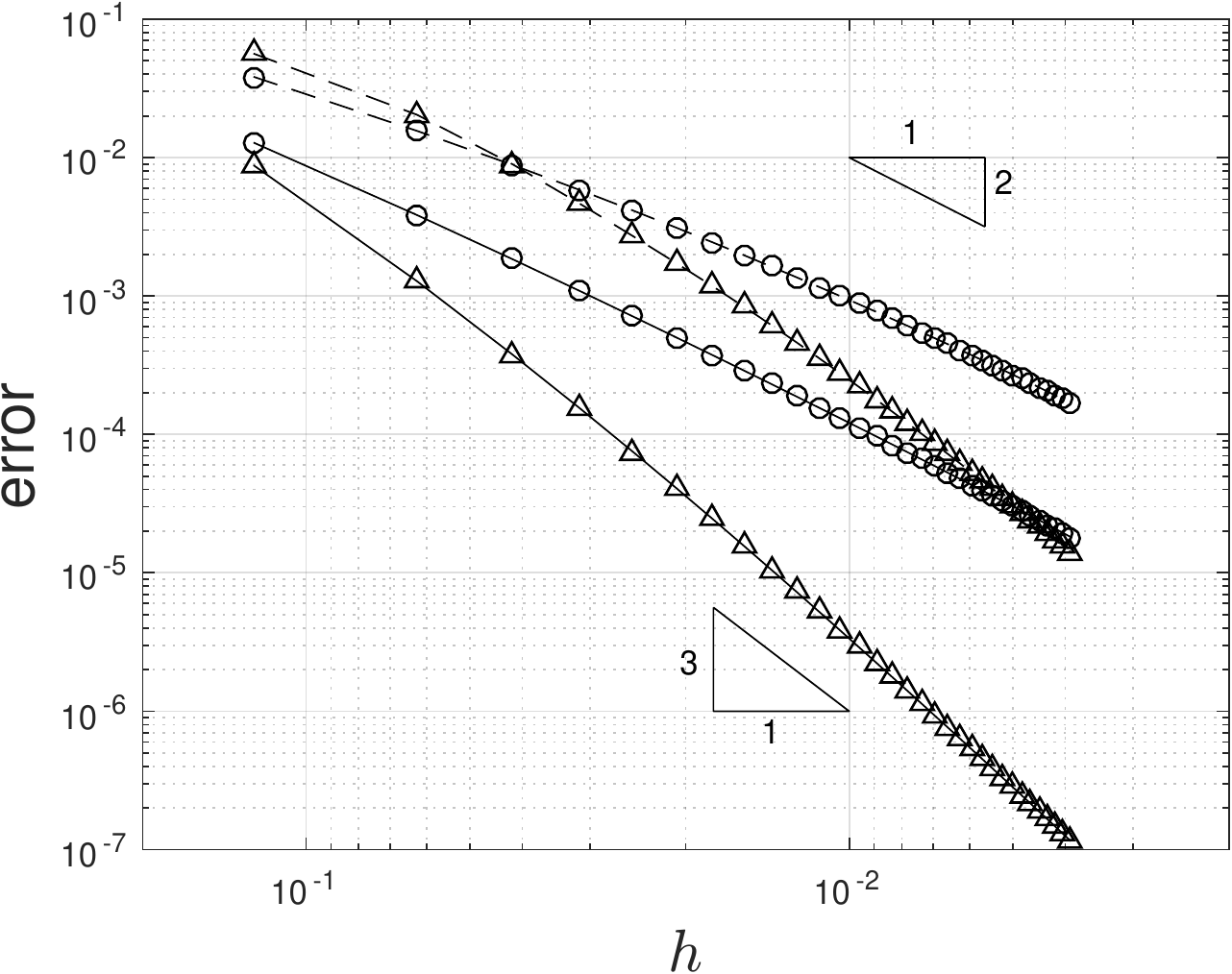}}
}
\caption{Relative error of the electrode potential $U$ as a function of the mesh size parameter $h$ for the two conductance models and two conductivity/conductance ratios. Left:\ Piecewise linear FEM. Right:\ Quadratic FEM. The dashed (i.e., upper) curves correspond to a lower ratio $\sigma / \zeta_\text{el}$ that is closer to the shunt model.}
\label{fig:uconv}
\end{figure}

The left-hand plot in Figure~\ref{fig:uconv} demonstrates that when using piecewise linear finite elements, the smoothened model converges 
faster toward its reference solution, but the difference cannot be considered significant.
Both models seem to exhibit an asymptotic decay rate of $h^{2}$ or a bit less. For the standard CEM, this is in line with the material in \cite{Darde16}, where the rate $h^{2 - \epsilon}$, $\epsilon > 0$, is predicted; for the smoothened hat-CEM, one would expect to gain the aforementioned $\epsilon > 0$ that is due to the solution of the standard CEM only lying in $\mathcal{H}^{2- \epsilon}$.
On the other hand, there is a big difference between the two conductance values:\ the further the values are from the shunt model (i.e., from $\zeta= \infty$), the more accurate the numerical solutions are for a given mesh. This observation is also in tune with \cite{Darde16}. The dashed curves, which correspond to the values in~\cite{Darde13b}, clearly show the numerical difficulties that appear when approaching the shunt model.

There is a greater difference between the models when quadratic elements are used. The right-hand plot in Figure~\ref{fig:uconv} indicates that the asymptotic convergence rate of the smoothened model increases to approximately $h^{3}$, whereas the rate for the standard model stays at about $h^{2}$. In fact, an extrapolation of the argumentation in \cite{Darde16} suggests that the asymptotic rates should be $h^{3 - \epsilon}$ and $h^{2 - \epsilon}$, respectively, with the smoothness of the forward solutions ($\mathcal{H}^{3-\epsilon}$ and $\mathcal{H}^{2-\epsilon}$, respectively) imposing in both cases the upper bound on the speed of convergence. In particular, it is to be expected that the use of, say, third order elements would improve the asymptotic convergence rate only if the model for the contact conductance were also further smoothened; see Theorem~\ref{thm:smoothness} and \cite{Larson13}. 
Once again, more accurate solutions are obtained for the higher ratio $\sigma / \zeta_\text{el}$,
that is, when the contacts are worse and the setting is further away from the shunt model (cf.~\cite{Darde16}).

\begin{figure}
\center{
{\includegraphics[height=1.8in]{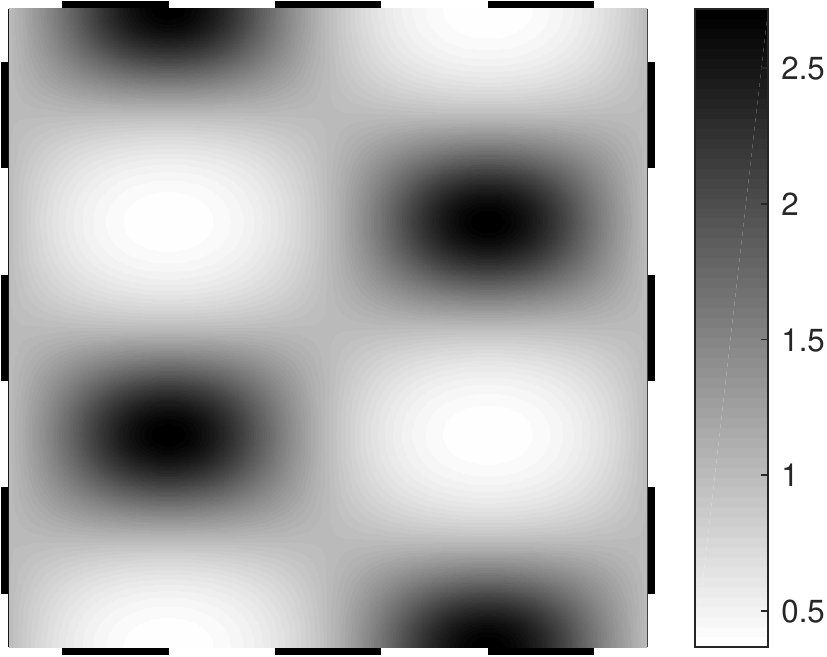}}
\qquad
{\includegraphics[scale=.4]{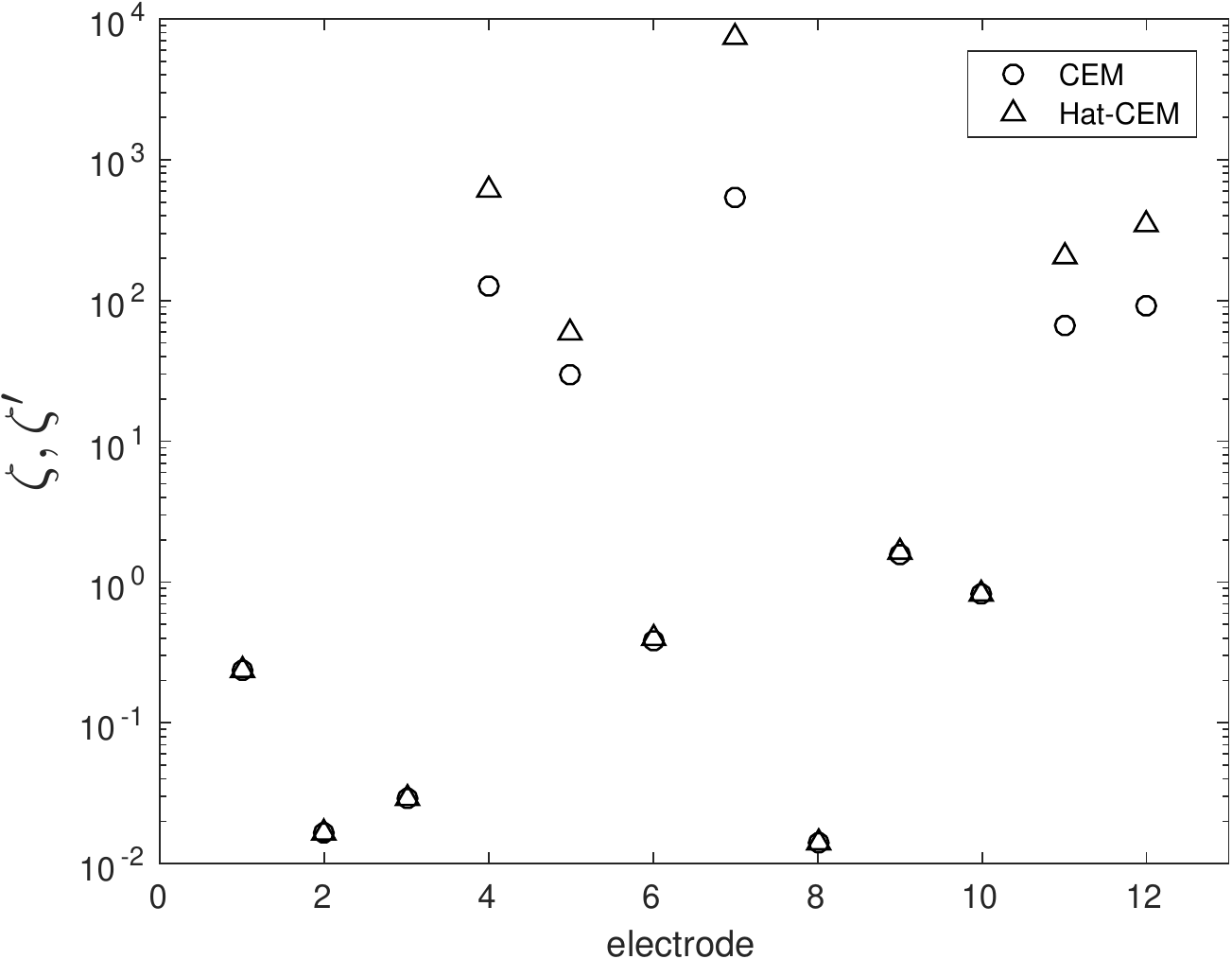}}
}
\caption{The conductivity phantom (left) and the contact conductance parameters (right) for the convergence tests of Figure~\ref{fig:uconv2}. The circles depict the conductances for the traditional CEM, the triangles the half-heights of the corresponding hat-shaped conductances for the smoothened CEM (cf.~Figure~\ref{fig:potdiff}). The electrodes are shown as thick line segments and they are numbered counter-clockwise starting from the bottom left corner.}
\label{fig:phantom}
\end{figure}

To confirm the above findings about the convergence of FEM, we repeat the numerical experiment with the constant conductivity replaced by the phantom shown on left in Figure~\ref{fig:phantom}. This time there are twelve electrodes attached to the boundary of the unit square, and the contact conductances shown on right in  Figure~\ref{fig:phantom} are chosen randomly so that their ratios with the mean of the conductivity phantom cover approximately the whole scale on the horizontal axes of the images in Figure~\ref{fig:potdiff}. Figure~\ref{fig:uconv2} presents the corresponding results. It is organized in the same way as Figure~\ref{fig:uconv}, that is, the left-hand image illustrates the convergence for the two models with piecewise linear FEM, whereas the right-hand image considers quadratic FEM. The conclusions are the same as for the first experiment: For piecewise linear basis functions both models exhibit convergence rates of approximately $h^2$, with the smoothened CEM converging slightly faster asymptotically. For the quadratic basis functions, the smoothened model clearly prevails with a decay rate $h^3$ compared to $h^2$ for the traditional model.

\begin{figure}
\center{
{\includegraphics[scale=.4]{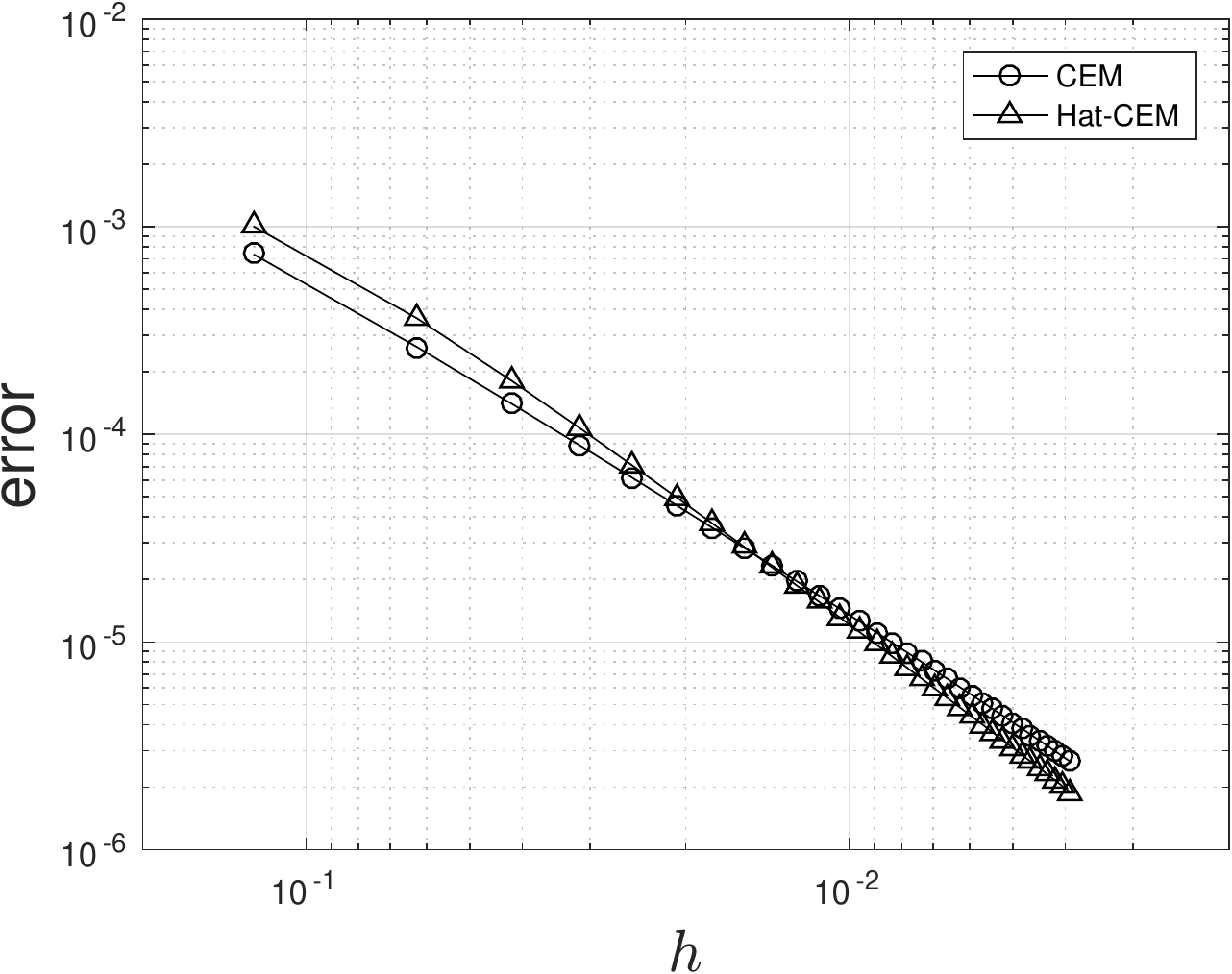}}
\qquad
{\includegraphics[scale=.4]{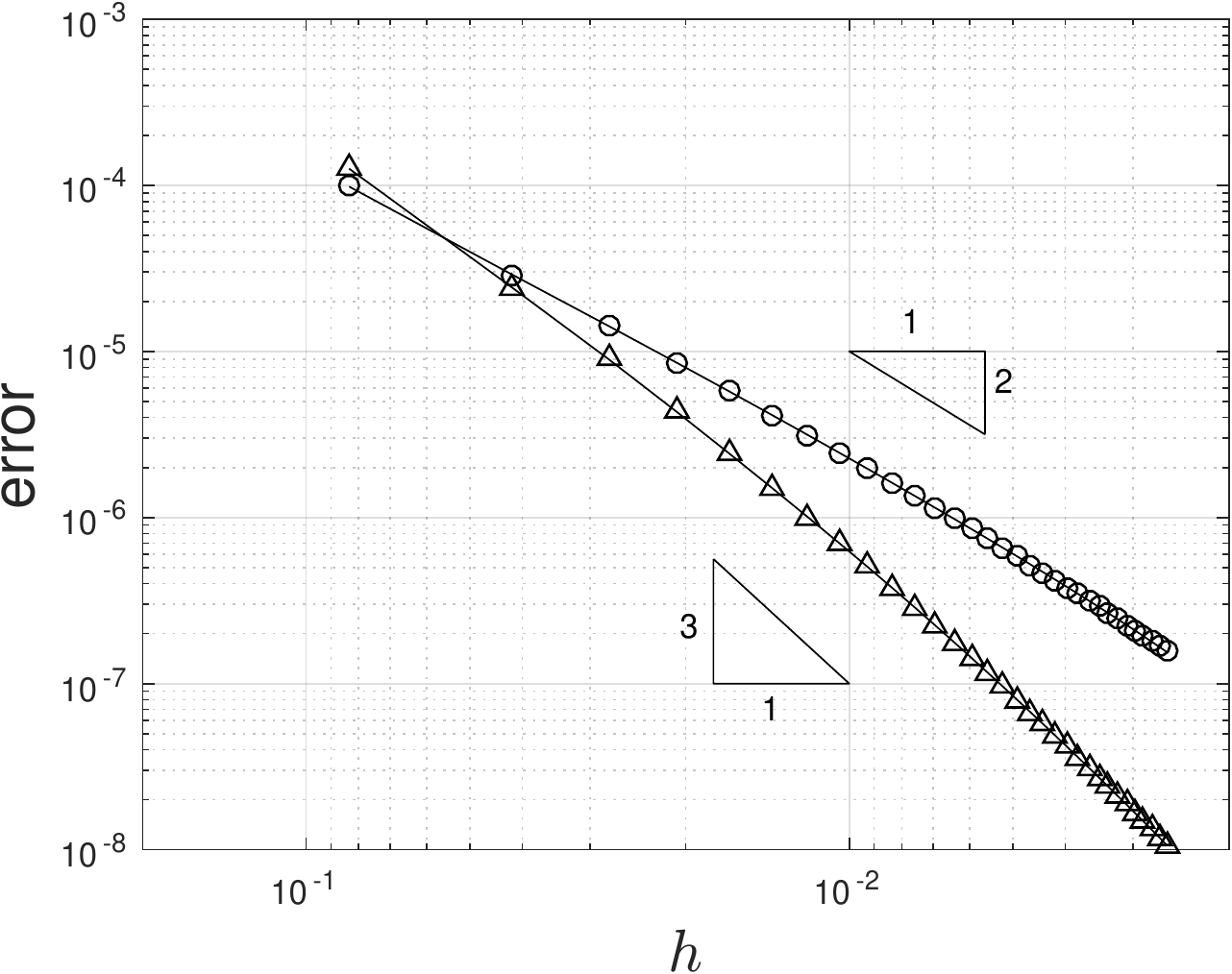}}
}
\caption{Relative error of the electrode potential $U$ as a function of the mesh size parameter $h$ for the two conductance models and for the conductivity phantom and contact conductance parameters shown in Figure~\ref{fig:phantom}. Left:\ Piecewise linear FEM. Right:\ Quadratic FEM.}
\label{fig:uconv2}
\end{figure}

The difference between the standard CEM and the new smoothened version becomes apparent even for the piecewise linear FEM when studying the convergence of the numerical Fr\'echet derivatives.
To demonstrate this, we consider the integrals \eqref{eq:sampling1}--\eqref{eq:sampling3} with the choices $\sigma = 1$ and $h_\nu = \|h_\tau\|_2  = 1$
as well as $\kappa|_E = 0$ dictated by the geometry.
More precisely, we numerically evaluate the integrals
\begin{align}
\label{eq:newsampling1}
\mathcal{I}_1^{(m,n)} &:= \int_{\partial \Omega} \zeta^2 (U^{(m)}-u^{(m)}) (U^{(n)}-u^{(n)}) \, {\rm d}S, \\
\label{eq:newsampling2}
\mathcal{I}_2^{(m,n)} &:= \int_{\partial \Omega} 
\dot{\zeta} (U^{(m)}-u^{(m)}) (U^{(n)}-u^{(n)}) \, {\rm d}S, \\
\label{eq:newsampling3}
\mathcal{I}_2^{(m,n)} &:= \int_{\partial \Omega} (\nabla u^{(m)})_\tau \cdot (\nabla u^{(n)})_\tau \, {\rm d}S
\end{align}
based on approximations of the solutions to \eqref{eq:cemeqs} on different finite element meshes. Here $\dot{\zeta}$ denotes the derivative of the contact conductance with respect to the arc\-length parameter. Notice that $\dot{\zeta}$ becomes a linear combination of delta distributions and $\mathcal{I}_2^{(m,n)}$ a linear combination of pointwise evaluations of the integrand $(U^{(m)}-u^{(m)}) (U^{(n)}-u^{(n)})$ at the end points of the electrodes when the traditional CEM is considered \cite{Darde13a,Darde13b}.

\begin{figure}
\center{
{\includegraphics[scale=.4]{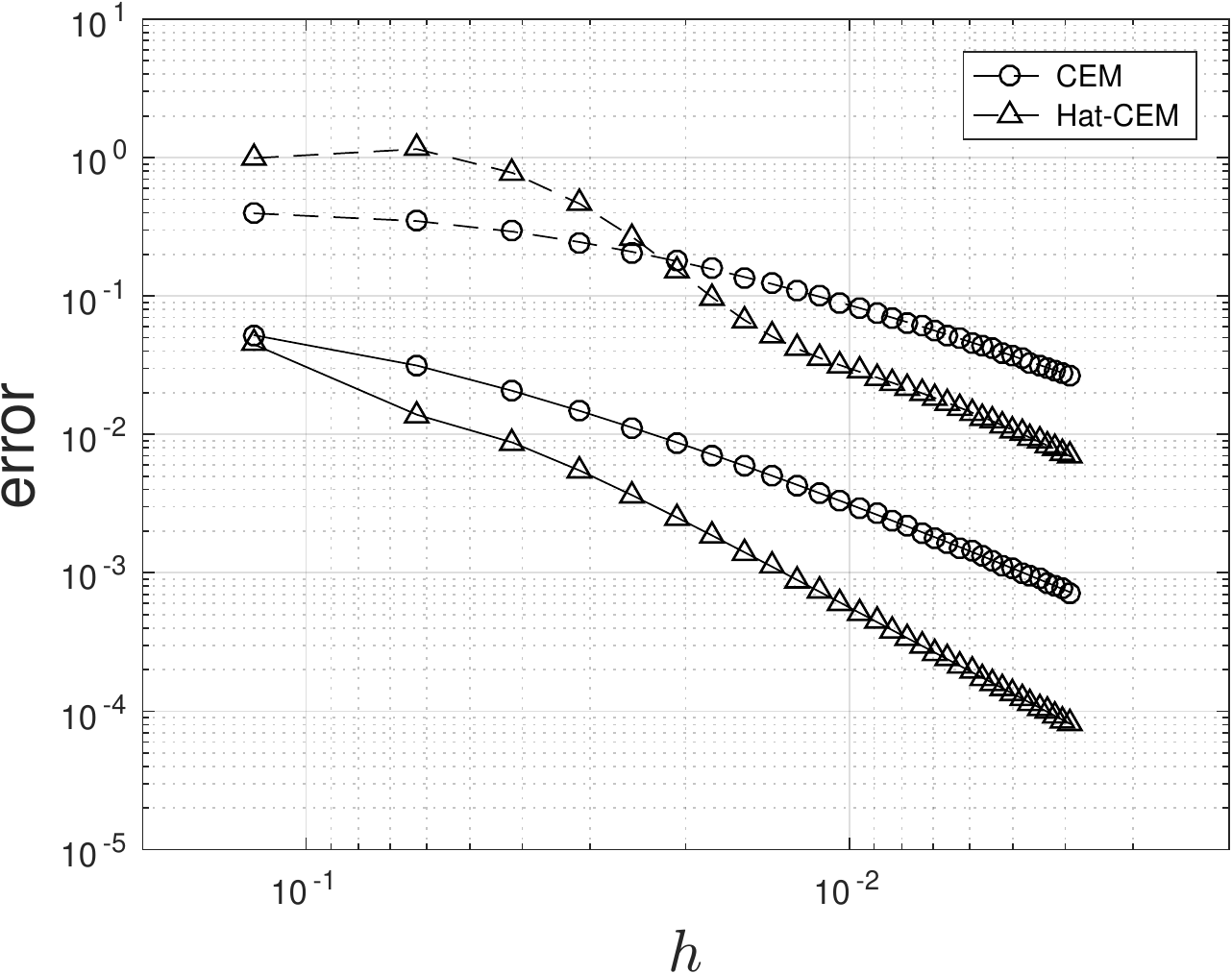}}
}
\center{
{\includegraphics[scale=.4]{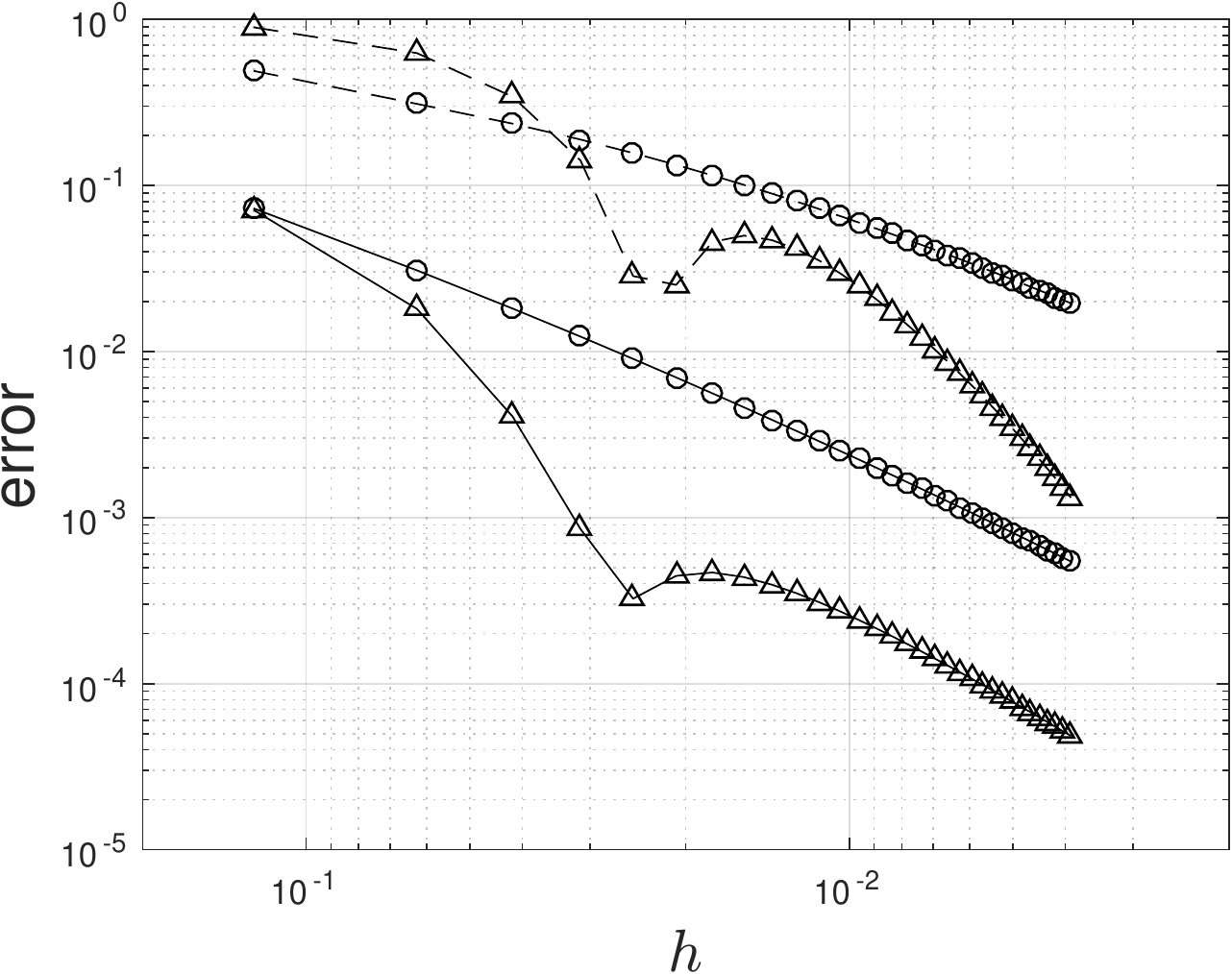}}
\qquad
{\includegraphics[scale=.4]{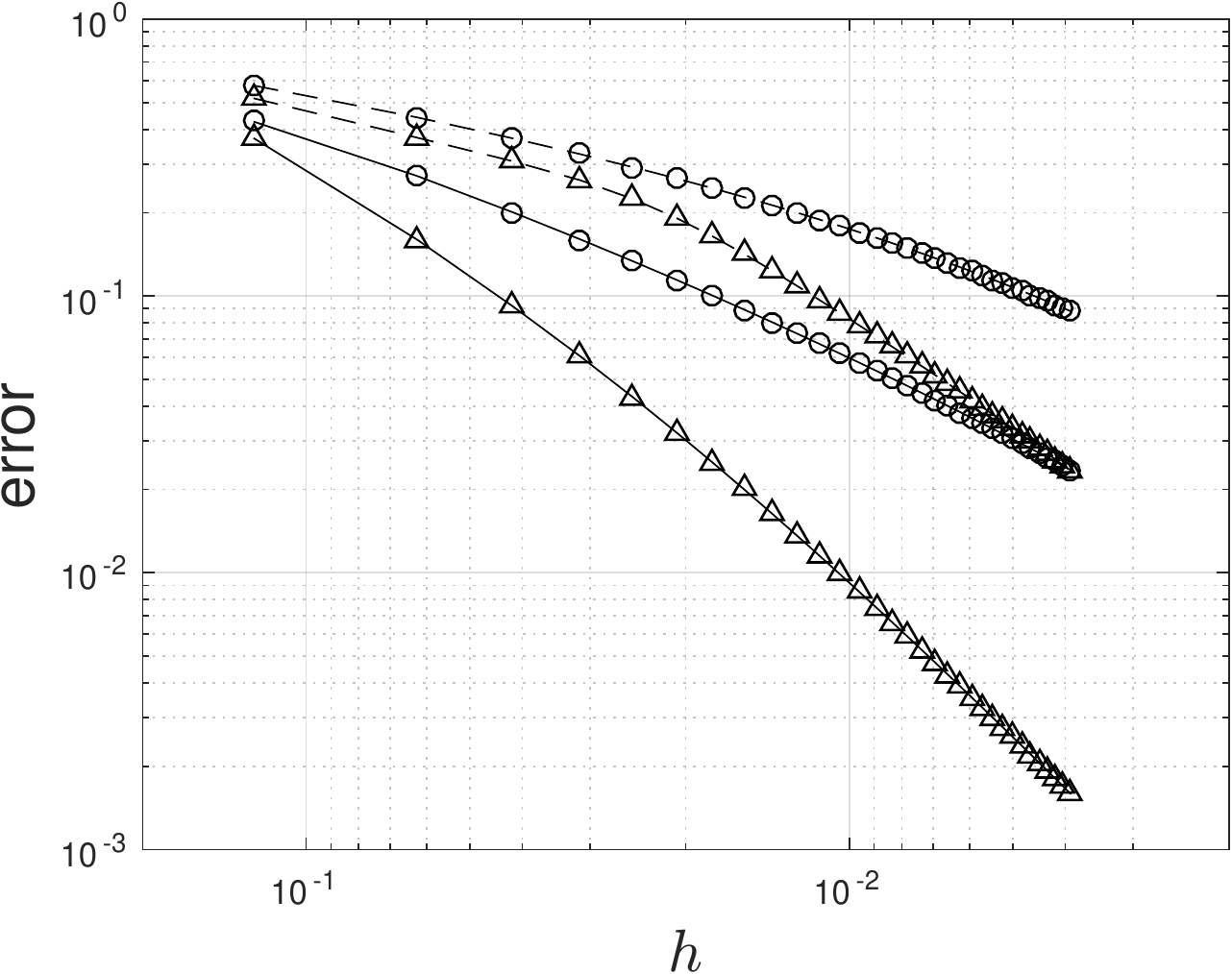}}
}
\caption{Relative errors of the three shape derivative integrals as functions of the mesh size parameter $h$ for the two conductance models and two conductivity/conductance ratios. Top: integral \eqref{eq:newsampling1}. Bottom left: integral \eqref{eq:newsampling2}. Bottom right: integral \eqref{eq:newsampling3}. The dashed (i.e., upper) curves correspond to a lower ratio $\sigma / \zeta_\text{el}$ that is closer to the shunt model. All solutions are computed by using piecewise linear elements.}
\label{fig:dconv}
\end{figure}

We only consider the setup with eight electrodes depicted in Figure~\ref{fig:setup} as well as a constant conductivity and two sets of identical contact conductances so that either $\sigma / \zeta_\text{el} \approx 50 \cdot 10^{-3} \,\mathrm{m}$ or $\sigma / \zeta_\text{el} \approx 4 \cdot 10^{-3} \,\mathrm{m}$, which correspond to $\sigma/\zeta_\text{el}' \approx 30 \cdot 10^{-3} \,\mathrm{m}$ or $\sigma/\zeta_\text{el}' \approx 0.5 \cdot 10^{-3} \,\mathrm{m}$, respectively (cf.~the bottom image of Figure~\ref{fig:potdiff}). In other words, the parameter choices are as in Figure~\ref{fig:uconv}.  The corresponding results for the inhomogeneous phantom and the contact conductances of Figure~\ref{fig:phantom} would be analogous.
The employed current patterns and their correspondence to the forward solutions of \eqref{eq:cemeqs} in \eqref{eq:newsampling1}--\eqref{eq:newsampling3} are as before, and the relative errors are computed via
\begin{equation*}
\delta_i := \left( \sum_{m=1}^{M-1} \sum_{n=1}^m \left\lvert \mathcal{I}_i^{(m,n)} - \tilde{\mathcal{I}}_i^{(m,n)} \right\rvert^2 \right)^{1/2} \Bigg/
       \left( \sum_{m=1}^{M-1} \sum_{n=1}^m \left\lvert \tilde{\mathcal{I}}_i^{(m,n)} \right\rvert^2 \right)^{1/2},
\end{equation*}
where tilde denotes an integral computed using the appropriate reference solutions on the densest mesh with $2049^2$ nodes. The results for $\delta_i, i=1,2,3$ are plotted in Figure~\ref{fig:dconv}, which demonstrates that the smoothened model is considerably more accurate than the traditional CEM, except for very coarse meshes. In addition, the difference between the two conductance values is again clearly visible.

\subsection{Comparison to experimental EIT data}
\label{sec:reco0}

The compatibility of the proposed smoothened model with real measurements is studied with data from a water tank shown on left in Figure~\ref{fig:reco0}. There are 16 electrodes of width $2 \,\mathrm{cm}$ attached to the interior lateral surface of the tank, extending from the bottom all the way up to the water surface. The circumference of the tank is $106 \,\mathrm{cm}$. The actual measurements were performed with low-frequency (1\,kHz) alternating current using the {\em Kuopio impedance tomography} (KIT4) device~\cite{Kourunen09}. The phase information is ignored and the data interpreted as if resulting from the use of direct current, which is reasonable due to the insignificance of capacitive effects at low temporal frequencies \cite{Vauhkonen97}, as explained at the beginning of Section~\ref{sec:numerics}.
The geometry, which is essentially two-dimensional because the electrodes are of the same height as the water layer (see, e.g., \cite{Leinonen14}), is known up to unavoidable mismodeling due to small imperfections in the construction of the tank. The domain is discretized for the potential with triangular elements having $12\,218$ nodes. Each electrode is divided into approximately $10$ element edges.

We make the reasonable assumption that the conductivity of the water layer within the tank is homogeneous and aim at reconstructing its value and the contact conductance parameters by solving the minimization problem
\begin{equation}
\label{eq:lsquares}
\argmin_y \, \lVert \mathcal{U}(y) - \tilde{\mathcal{U}} \rVert_2^2
\end{equation}
where $y \in \R_+^{1+16}$ represents the unknowns. The vector $\tilde{\mathcal{U}} \in \R^{240}$ contains all measured electrode potentials and the values in $\mathcal{U}(y)$ are the computed potentials for the employed fifteen linearly independent current patterns and a given parameter vector $y$. For the traditional CEM, the latter sixteen components of $y$ are the contact conductances, whereas for the smoothened CEM they represent the half-heights of the hat-shaped conductance functions. 

The employed minimization algorithm is based on Levenberg--Marquardt method.
The initial guesses for the constant conductivity and contact conductances were chosen as $\sigma=0.25\, \mathrm{mS}/\mathrm{cm}$ and $\zeta_\text{el}=10\, \mathrm{mS}/\mathrm{cm}^2$, respectively, for the traditional CEM. For the smoothened CEM, the values $\sigma=0.25\, \mathrm{mS}/\mathrm{cm}$ and $\zeta_\text{el}'=70\, \mathrm{mS}/\mathrm{cm}^2$ were used. The chosen initial values for $\sigma$ and $\zeta_\text{el}$ are close to the ones used in \cite{Hyvonen17} and the ratio $\zeta_\text{el} / \zeta_\text{el}'$ agrees with the appropriate value on the graph in the bottom image of Figure~\ref{fig:potdiff}. 
In practice, one need not have such a graph available since the minimization process does not seem to be sensitive to the initial guess for $\zeta_\text{el}'$; this conclusion also  applies to the considerations in Section~\ref{sec:reco}.
The three-dimensional quantities for the conductivity and the conductance can be converted to their two-dimensional counterparts by multiplying with the height of the tank which is $5\; \mathrm{cm}$.
For both models, the minimization algorithm converged without any complications.

\begin{figure}
\center{
{\includegraphics[scale=.215]{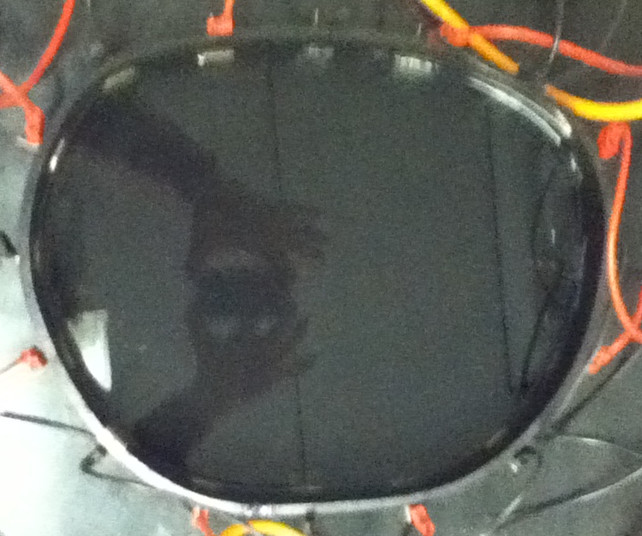}}
\qquad
{\includegraphics[scale=.4]{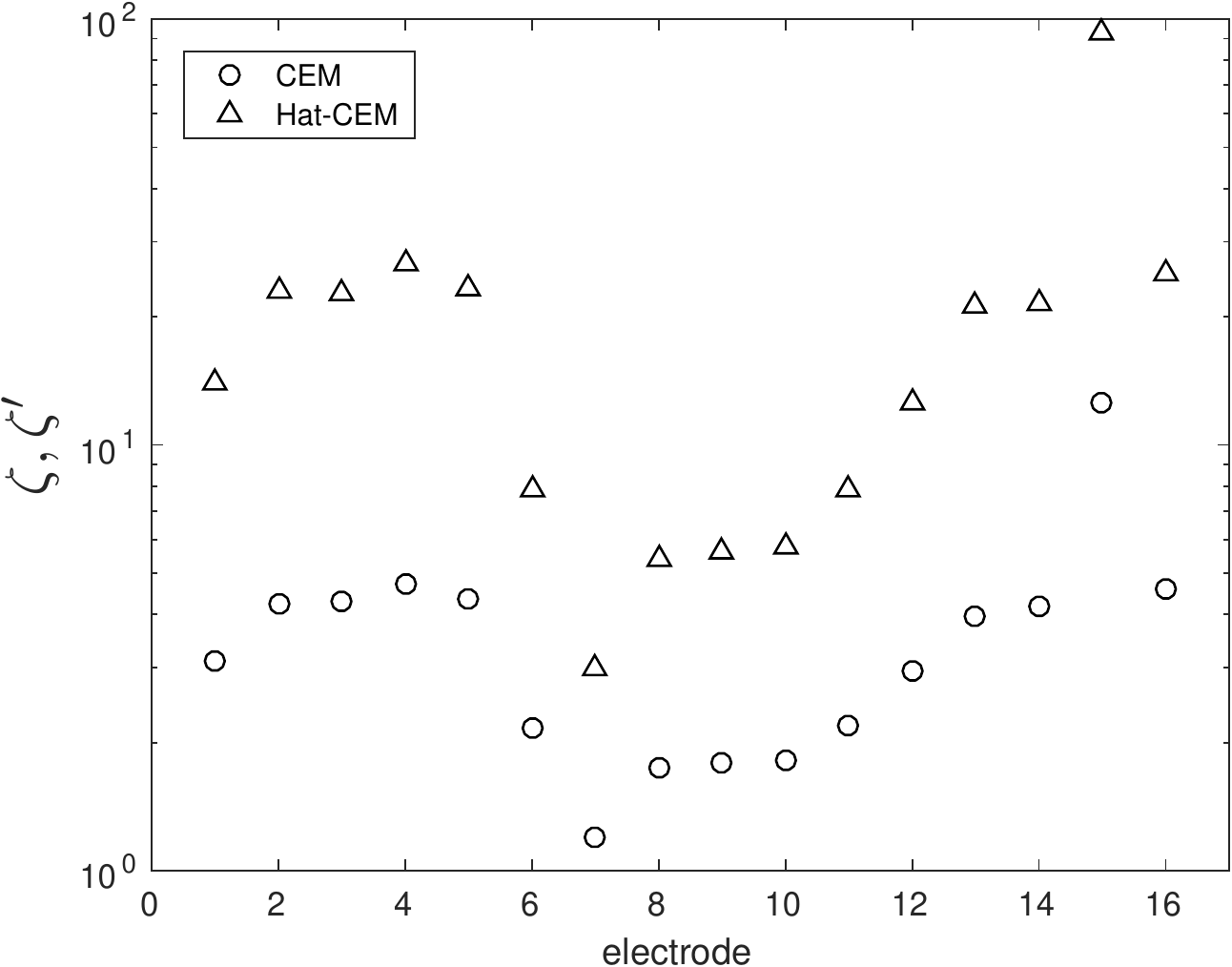}}
}
\caption{Left:\ Photo of the tank filled with Finnish tap water. Right:\ Contact conductance parameters giving the best match with the data. The circles depict the conductances for the traditional CEM, the triangles the half-heights of the hat-shaped conductances for the smoothened CEM. The unit of conductance is $\mathrm{mS}/\mathrm{cm}^2$.}
\label{fig:reco0}
\end{figure}

The optimal values for the constant conductivity level are $\sigma = 0.22722\, \mathrm{mS}/\mathrm{cm}$ for the traditional model and $\sigma = 0.22720\, \mathrm{mS}/\mathrm{cm}$ for the smoothened one. The corresponding contact conductance parameters for the two models are shown on right in Figure~\ref{fig:reco0}. The results are well in line with the bottom image of Figure~\ref{fig:potdiff}: As the contacts are good in water tank experiments, i.e., the ratio $\sigma/\zeta$ is low, one expects the half-heights of the hat-like conductances to be substantially higher than the constant contact conductances for the traditional CEM. The minimal relative discrepancies between the measurements and the two models, i.e.,
\begin{equation*}
\frac{\lVert \mathcal{U}(y^*) - \tilde{\mathcal{U}} \rVert_2}{\lVert \tilde{\mathcal{U}} \rVert_2}
\end{equation*}
where $y^* \in \R^{17}$ is the respective solution of \eqref{eq:lsquares}, was $1.20 \%$ for the traditional CEM and $1.21 \%$ for the smoothened CEM with hat-like conductances. If one considers the discrepancy in comparison to the maximal variation in the measurements, i.e., replaces $\lVert \tilde{\mathcal{U}} \rVert_2$ in the denominator by  $\sqrt{240}\, \max_{1 \leq i,j \leq 240} \lvert \tilde{\mathcal{U}}_i-\tilde{\mathcal{U}}_j \rvert$,
the numbers are $0.220\%$ for the traditional model and $0.221\%$ for the smoothened one. In particular, according to this single experiment, the two models seem to be in approximately as good accordance with real-world EIT data.

The relative discrepancies listed above are somewhat higher than the expected noise level in the data (cf.~\cite{Kourunen09}). Part of this extra mismatch probably originates from small errors in the model for the measurement configuration; absolute EIT is known to be extremely sensitive to geometric mismodeling~\cite{Kolehmainen97}. In addition, treating the data as if it originated from direct current measurements may have a small effect on the results. In particular, the evidence about the validity of the smoothened model presented here is not conclusive, but it needs to be confirmed by more carefully designed experimental studies in the future.

\subsection{Reconstructions from experimental EIT data}
\label{sec:reco}
Let us next consider the same water tank as in the previous section, but this time with one embedded insulating cylindrical inclusion of radius $3.5\, {\rm cm}$ made out of plastic; see the top image in Figure~\ref{fig:reco}.
This time the (two-dimensional) domain is discretized 
for the potential and the conductivity
with triangular elements having $8004$ nodes. 
Each electrode is again divided into approximately $10$ element edges. Our goal is to compute reconstructions of the conductivity phantom by using the traditional and the smoothened CEM and to demonstrate that the choice between the two models does not have a significant effect on the outcome. 
In particular, our objective is not to exploit the faster convergence of FEM approximations for the smoothened model by employing sparser meshes or/and higher order elements, but such considerations are left for future studies.

The reconstruction algorithm aims at computing a MAP estimate
\begin{equation}
\label{eq:MAP}
\argmin_y \left\{ \lVert \mathcal{U}(y) - \tilde{\mathcal{U}} \rVert_2^2 + \lVert G (y-y_0) \rVert_2^2 \right\}
\end{equation}
for the unknown parameter vector $y \in \R_+^{8004+16}$ that represents the discretized conductivity field and the contact conductance values and whose expected value is $y_0$.
As in the previous section, the vector $\tilde{\mathcal{U}} \in \R^{240}$ contains all measured electrode potentials and the values in $\mathcal{U}(y)$ are the computed potentials for the employed fifteen linearly independent current patterns and a given parameter vector $y$.
The positive semidefinite matrix $G$ originates from combining a prior distribution for the conductivity with the assumption that each measured electrode potential is corrupted by an independent realization of a normally distributed random variable with zero mean and standard deviation $2 \cdot 10^{-3} \max_{1 \leq i,j \leq 240} \lvert \tilde{\mathcal{U}}_i-\tilde{\mathcal{U}}_j \rvert$, 
the choice of which is motivated by the relative discrepancies listed in the previous section.
More precisely, it is formally assumed that the conductivity is {\em a~priori} a Gaussian random field with correlation length $4\,\mathrm{cm}$, pointwise standard deviation $0.25\,\mathrm{mS}/\mathrm{cm}$, and a constant expectation function $0.25\,\mathrm{mS}/\mathrm{cm}$ that is close to the conductivity of Finnish tap water. 
The correlation length is intentionally chosen to be of the same order as the radius of the inclusion, i.e., it corresponds to the size of inhomogeneities we expect to find inside the tank. Choosing a longer correlation length would blur the reconstruction of the inclusion, whereas employing a significantly shorter correlation length would lead to unwanted oscillations in the reconstruction of the background conductivity level.
The contact conductances are estimated without further prior knowledge and thus the corresponding parts of $G$ and $y_0$ are empty.
The employed minimization algorithm is again based on Levenberg--Marquardt method.
We refer to,~e.g.,~\cite{Darde13b,Kaipio05} for more details on Bayesian inversion.

\begin{figure}
\includegraphics[scale=.18]{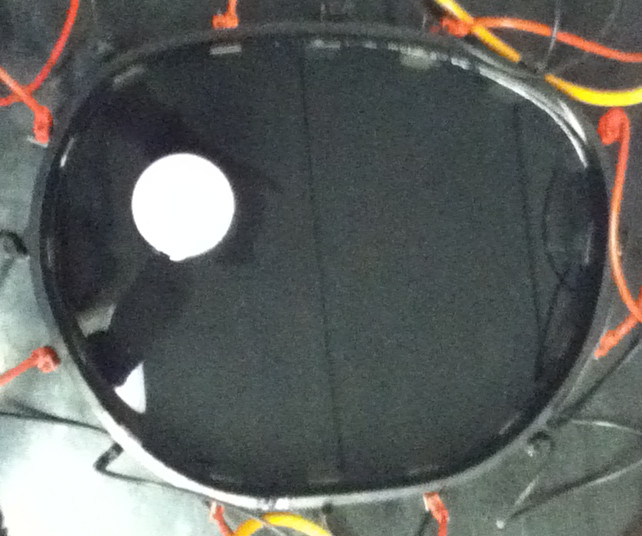}
\center{
{\includegraphics[scale=.35]{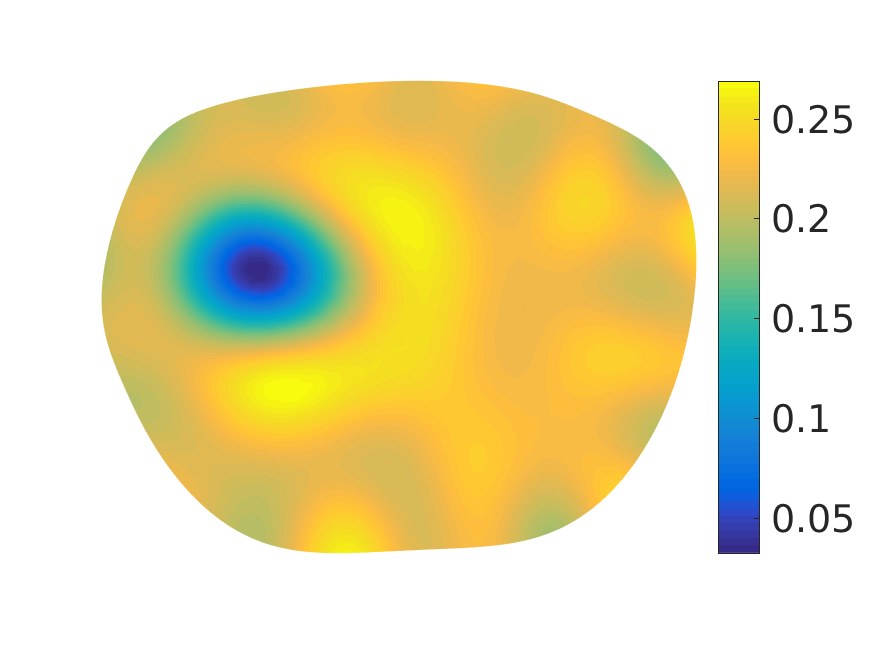}}
\qquad
{\includegraphics[scale=.35]{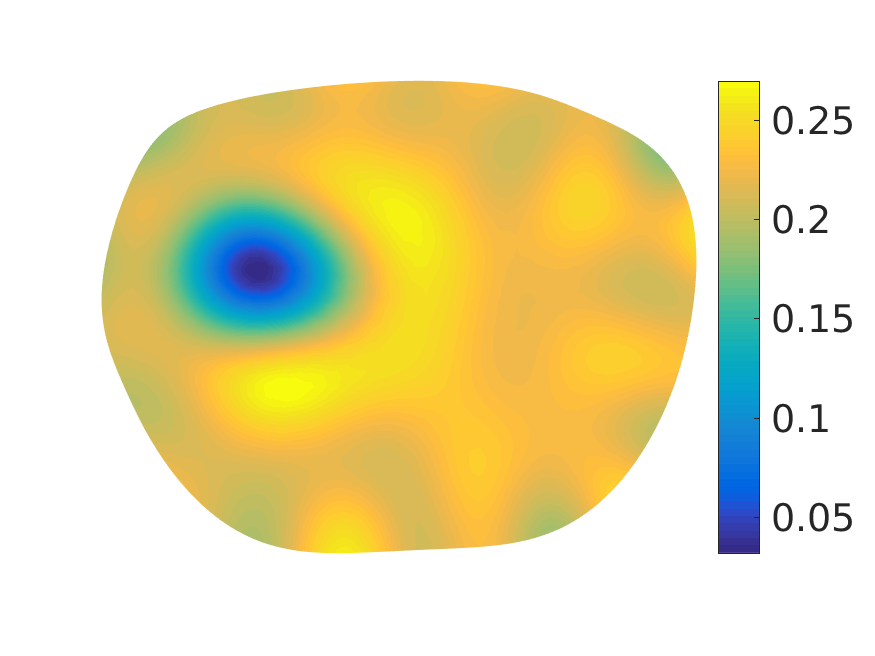}}
}
\caption{Top:\ Photo of the water tank with one insulating inhomogeneity. Left:\ Reconstruction based on the traditional CEM. Right:\ Reconstruction using the proposed hat function model for the contact conductances. The unit of conductivity is $\mathrm{mS}/\mathrm{cm}$.}
\label{fig:reco}
\end{figure}

The initial values for the conductivity and the contact conductance parameters are the same as in the previous section.  In particular, the initial guess/expected value for the conductivity is homogeneous.
The reconstruction corresponding to the traditional CEM is presented on the bottom left in Figure~\ref{fig:reco} and the one corresponding to the smoothened model with hat-like conductances on the bottom right.
The results are remarkably similar, 
which is what we have observed with other water tank experiments as well. 
In fact, the reconstruction obtained by using the smoothened model is practically indistinguishable from the one corresponding to the traditional CEM: their relative discrepancy in the $L^2(\Omega)$ norm is just $0.3 \%$.
Both reconstructions also clearly indicate the location of the insulating inclusion as a region of almost vanishing conductivity and the background conductivity level is close to the values obtained in the previous section. 
Thus, the proposed model seems to be as compatible with experimental measurements as the traditional CEM, at least in the examined water tank setting where the contact conductances are relatively high. The reason may simply be that the difference between the two models is quite small when the conductance functions are scaled properly, as demonstrated by the top right plot in Figure~\ref{fig:potdiff}. On the other hand, it cannot be ruled out that the real physical phenomena at the electrode contacts may even be more accurately described by some variant of the proposed smoothened CEM.

\section{Concluding remarks}
\label{sec:conclusion}
We have introduced a smoothened version of the CEM for EIT.
The new model retains the essential solvability and differentiability properties, while the regularity of the solution can be arbitrarily improved by choosing an appropriate smoothened conductance for the electrode contacts.
It was numerically demonstrated that at least the simplest, piecewise linear smoothening is almost equivalent to the standard CEM, if the hat-shaped conductance functions are scaled properly. What is more, the presented EIT reconstruction with the proposed model is almost indistinguishable from the one obtained with the traditional CEM. The computational feasibility of the new model was also demonstrated by the superior convergence of the FEM. 

To summarize, we recommend using some variant of the smoothened CEM especially if one wants to exploit a higher-order FEM solver or needs to compute numerical shape derivatives. However, even the proposed smoothened model cannot overcome the numerical problems related to very high contact conductances that appear when the measurement setting approaches the so called shunt model.

% Bibliography using bibtex
\bibliographystyle{acm}
\bibliography{smoothcem-refs}

\end{document}